\documentclass[
  a4paper, 
  reqno, 
  oneside, 
  12pt
]{amsart}

\usepackage[utf8]{inputenc}

\usepackage[dvipsnames]{xcolor} 
\colorlet{cite}{red}
\usepackage{tikz}  
\usetikzlibrary{arrows,positioning}
\tikzset{ 
  baseline=-2.3pt,
  text height=1.5ex, text depth=0.25ex,
  >=stealth,
  node distance=2cm,
  mid/.style={fill=white,inner sep=2.5pt},
}

\usepackage{lmodern}
\usepackage{tikz-cd}
\usepackage{amsthm, amssymb, amsfonts}

\usepackage{amsthm, amssymb, amsfonts}	
\usepackage[all]{xy}
\usepackage{graphicx,caption,subcaption}
\usepackage{braket}  
\usepackage{microtype}
\usepackage[makeroom]{cancel}
\usepackage[%
  bookmarks=true,			
  unicode=true,			
  pdftitle={Infinitesimally Tight Lagrangian Orbits},		%
  pdfauthor={Gasparim}{San Martin}{Valencia},	%
  pdfkeywords={Infinitesimally Tight Lagrangian Orbits},	
  colorlinks=true,		
  linkcolor=Red,			
  citecolor=cite,		
  filecolor=magenta,		
  urlcolor=RoyalBlue			
]{hyperref}				
\usepackage{cleveref}

\newtheoremstyle{mydef}
  {}		
  {}		
  {}		
  {}		
  {\scshape}	
  {. }		
  { }		
  {\thmname{#1}\thmnumber{ #2}\thmnote{ #3}}	

\newtheorem{theorem}{Theorem}[section]
\newtheorem*{theorem*}{Theorem}
\newtheorem{proposition}[theorem]{Proposition}
\newtheorem*{proposition*}{Proposition}
\newtheorem{lemma}[theorem]{Lemma}
\newtheorem*{lemma*}{Lemma}
\newtheorem{corollary}[theorem]{Corollary}
\newtheorem*{corollary*}{Corollary}
\theoremstyle{definition}
\newtheorem{definition}[theorem]{Definition}
\newtheorem{example}[theorem]{Example}

\theoremstyle{remark}
\newtheorem{remark}[theorem]{Remark}

\newtheorem*{problem*}{Problem}
\usepackage{tikz}

\DeclareMathOperator{\im}{Im}

\DeclareMathOperator{\ad}{ad}
\DeclareMathOperator{\Ad}{Ad}

\author{Elizabeth Gasparim, Luiz A. B. San Martin, Fabricio Valencia}
\address{}
\date{\today}

\address{E. Gasparim - Depto. Matem\'aticas, Univ. Cat\'olica del Norte, Antofagasta, Chile, \newline  
      L. A. B. San Martin - Depto. de Matem\'{a}tica, Imecc - Unicamp, Campinas, Brasil,\newline
F. Valencia - Inst. Matem\'aticas, Univ. de Antioquia, Medell\'in, Colombia. \newline
	etgasparim@gmail.com,  smartin@ime.unicamp.br, fabricioyarro@gmail.com
}
\title[Infinitesimally Tight Lagrangian Orbits]{Infinitesimally Tight Lagrangian Orbits}

\begin{document}
\maketitle

\begin{abstract}
We describe isotropic orbits for the restricted action of a subgroup of a Lie group acting on a symplectic manifold by Hamiltonian symplectomorphisms and admitting an Ad*-equivariant moment map. We obtain examples of Lagrangian orbits of complex flag manifolds, of cotangent bundles of orthogonal Lie groups, and of products of flags. We introduce the notion of infinitesimally tight and study the intersection theory of such Lagrangian orbits, giving many examples.
\end{abstract}

\tableofcontents
\section{Introduction}
\label{intro}

Let $(M,\omega)$ be a connected symplectic manifold, $G$ a Lie group with Lie algebra $\mathfrak{g}$, and $L$ a Lie subgroup of $G$. Assume that there exists a Hamiltonian action of $G$ on $M$ which admits an $\Ad^*$-equivariant moment map $\mu:M\to\mathfrak{g}^*$. The purpose of this paper is to study those orbits $Lx$ with $x\in M$ that are Lagrangian submanifolds of $(M, \omega)$, or more generally, isotropic submanifolds. We also discuss some essential features  of the intersection theory of such  Lagrangian orbits, namely the concepts of locally tight 
and infinitesimally tight Lagrangians. 
The famous  Arnold--Givental conjecture, proved in many cases,  predicts that the number of intersection points of a Lagrangian $\mathcal{L}$ 
and its image $\varphi(\mathcal{L}) $ by the flow of a Hamiltonian vector field can be estimated from below by the sum of its 
$\mathbb Z_2$ Betti numbers:
$$|\mathcal{L}\cap \varphi(\mathcal{L})| \geq \sum b_k(\mathcal{L};\mathbb Z_2).$$
The concepts of tightness address those Lagrangians which attain the lower bound, and are therefore 
of general interest in symplectic geometry.\\ 

For us  additional  motivation to study Lagrangians and their intersection theory comes from 
questions related to the Homological Mirror Symmetry conjecture and in particular from 
concepts of objects and morphisms in the so called Fukaya--Seidel categories, which 
are generated by Lagrangian vanishing cycles (and their thimbles) 
with prescribed behavior  inside of symplectic fibrations. 
In  \cite[Thm.\thinspace 2.2]{GGSM1} it was shown that the 
usual height function from Lie theory 
gives  adjoint orbits of semisimple Lie groups the structure of symplectic Lefschetz fibrations. 
These give rise to what is known as  Landau--Ginzburg (LG) models.
We wish to study the Fukaya--Seidel category of these LG models. Finding Lagrangian submanifolds 
and understanding their intersection theory inside a  compactification is an initial  tool 
to investigate possible thimbles. The Fukaya--Seidel category of the LG model for the adjoint orbit of $\mathfrak{sl}(2,\mathbb{C})$ was calculated in \cite{BBGGSM} and such LG models was shown to have no projective mirrors (Theorems 4.1 and 7.6).\\
Products of flag manifolds occur as compactifications of adjoint orbits of semisimple
noncompact Lie groups, see \cite[Sec. 3]{GGSM2}, and this originated our particular interest in
finding Lagrangian submanifolds inside products of flags. 
Moreover,  minimal noncompact semisimple orbits  were shown in \cite{BGSMR} to satisfy the KKP conjecture, 
but a verification of the KKP conjecture for general semisimple orbits, which remains to be done,   
would also require better
understanding of the Lagrangians inside their compactification. 
Hence we have various motivations to search for  Lagrangians inside noncompact adjoint orbits, 
the compact ones, that is, the flag manifolds, 
and  products.\\ 

This paper is divided as follows. In Section \ref{isotropic}, we give a simple characterization of isotropic orbits. If $\mathfrak{l}$ denotes the Lie algebra of $L$ and $\mathfrak{l}'$ its derived algebra, we have
\begin{proposition*}{\bf \ref{anill}}
	An orbit $Lx$ is isotropic if and only if  $\mu \left(x\right) 
	$ belongs to the annihilator $\left( \mathfrak{l}^{\prime }\right) ^{\circ }$ of  $\mathfrak{l}^{\prime}$.
\end{proposition*}
In Section \ref{ortogonal}, we use Proposition \ref{anill} to characterize isotropic orbits  in the cotangent bundle of an orthogonal Lie group. In particular, if $G$ is a semisimple Lie group and $T^*(G)\approx G\times\mathfrak{g}^*$ is its cotangent bundle
we prove:
\begin{corollary*}{\bf \ref{SemisimpleCaseCotangent}}
	The only isotropic orbits by the natural left and right actions of $G$ on $T^\ast(G)$ are of the form $G(g,0)$ for all $g\in G$. Such orbits are Lagrangian and Hamiltonian isotopic to $G$.
\end{corollary*}
In Section \ref{secexflagcplx}, we consider compact semisimple Lie groups.
Endowing  adjoint orbits $\Ad(U)(iH_0)$ with the Kirillov--Kostant--Souriau symplectic form, we prove that the orbit of  a proper subgroup $L\subset U$  through the origin $iH_0$ is isotropic if and only if $\mathfrak{l}'\subset (iH_0)^{\perp}$. 
For example,  the orbit of $\mathrm{SO}(n)$ through the origin of any flag  of $SU(n)$ is Lagrangian.\\
An interesting example happens  when $U=\mathrm{SU}(3)$ and $L=U_H$ is the isotropy group in $U$ of the element $H=i\mathrm{diag}\{2,-1,-1\}$.
We have that $(\mathfrak{u}_H)^\perp$ intersects the 3 types of adjoint orbits of $SU(3)$, namely the flags $\mathbb{CP}^{2}$, $\mathrm{Gr}_{2}\left( 3,%
\mathbb{C}\right) $ and  $\mathbb{F}\left( 1,2\right) $. We prove that the only possible isotropic orbits  
of $L$ passing 
through $H$ are:
\begin{itemize}
	\item the trivial one, that is, a single point in $\mathbb{CP}^{2}$,
	\item a (2 dimensional) Lagrangian in $\mathrm{Gr}_{2}\left( 3,%
	\mathbb{C}\right) $, and 
	\item a (3 dimensional) Lagrangian in  the flag $\mathbb{F}\left( 1,2\right)$.
\end{itemize}
In Section \ref{secdiag}, we study Lagrangian orbits in  products of flag manifolds with 
respect to the diagonal and  shifted diagonal actions, showing:
\begin{theorem*}{\bf \ref{noniso}}
	A product of flags $\mathbb{F}_{\Theta _{1}}\times \mathbb{F}_{\Theta _{2}}$
	admits an isotropic orbit by the diagonal action if and only if \,
	$\mathbb{F}_{\Theta _{2}}$ is the dual  flag $\mathbb{F}_{\Theta
		_{1}^{\ast }}$ of \, $\mathbb{F}_{\Theta _{1}}$.
\end{theorem*}

Acting  by subgroups of the type $\Delta ^{m}=\{\left(
u,mum^{-1}\right)\in U\times U :u\in U\}$ we obtain:
\begin{proposition*}{\bf \ref{shifted11}}
	Inside the product $\mathbb{F}_{\Theta }\times \mathbb{F}_{\Theta ^{\ast }}$,
	for each $m\in U$,
	there exists a unique isotropic orbit of the diagonal action by the subgroup
	$\Delta ^{m}$. Such an orbit is 
	Lagrangian and it is given by the graph of the map $-\Ad\left( m\right) :%
	\Ad\left( U\right) \left( iH\right) \rightarrow \Ad\left(
	U\right) \left( i\sigma \left( H\right) \right)$.
\end{proposition*}
In particular, when $m=e$ is the identity in $U$, we prove that there exists a unique Lagrangian orbit of the diagonal action of $U$ on $\mathbb{F}_{\Theta }\times \mathbb{F}_{\Theta ^{\ast }}$  given as the graph of  $-\mathrm{id}:%
\Ad\left( U\right) \left( iH\right) \rightarrow \Ad\left(
U\right) \left( i\sigma \left( H\right) \right)$. Furthermore, as an important feature of the 
orbits by shifted diagonals is stated as:
\begin{theorem*}{\bf \ref{HamiltonianIsotopic}}
	All Lagrangian orbits in $\mathbb{F}_{\Theta }\times \mathbb{F}_{\Theta ^{\ast }}$ of Proposition \ref{shifted11} belong the same Hamiltonian isotopy class.
\end{theorem*}

Section \ref{imersions} is dedicated to the study of \emph{tight immersions}. We explore a new concept which we call \emph{infinitesimally tight} (Definition \ref{definfintight}). This notion is equivalent to the concept of locally tight  given by \cite{Oh1} (Definition \ref{tight}). In other words,
\begin{theorem*}{\bf \ref{lvi}}
	Let $G$ be a Lie group and $M$ a homogeneous space together with 
	a $G$-invariant symplectic form $\omega $.	Then a Lagrangian submanifold $\mathcal{L}\subset M$ is infinitesimally tight if and only if $\mathcal{L}$ is locally tight.
\end{theorem*}

As an example we show that the Lagrangian orbit  $S^{3}$ of \, $\mathrm{U}\left( 2\right) $ in the  flag $%
\mathbb{F}\left( 1,2\right) $ is infinitesimally tight.  In further generality, we obtain:
\begin{corollary*}{\bf \ref{shifted}}
	The  Lagrangian orbits  of type  
	$$\Gamma\left\lbrace-\Ad\left( m\right) :%
	\Ad\left( U\right) \left( iH\right) \rightarrow \Ad\left(
	U\right) \left( i\sigma \left( H\right) \right)\right\rbrace$$ corresponding to the  shifted diagonals $\Delta^m$  are  infinitesimally tight in $\mathbb{F}_{\Theta }\times \mathbb{F}_{\Theta ^{\ast }}$.
\end{corollary*}
In   Appendix A  we describe the KKS symplectic form on adjoint orbits of orthogonal Lie groups.
Finally, in Appendix B we give a list of open problems about Lagrangian orbits.

\section{Isotropic orbits}\label{isotropic}
Let $(M,\omega)$ be a connected symplectic manifold and $\cdot:G\times M\to M$ a {\it Hamiltonian} action of a Lie group $G$ on $M$. If $\mathfrak{g}$ is the Lie algebra of $G$ and $\mathfrak{g}^\ast$ its dual vector space, this means that the action is  symplectic and that there exists a smooth map $\mu:M\to \mathfrak{g}^\ast$, called {\it moment map}, such that for all $X\in\mathfrak{g}$
\begin{equation}\label{moment1}
\textnormal{d}\hat{\mu}(X)=\iota_{\widetilde{X}}\omega
\end{equation}\label{mu}
where $\hat{\mu}(X):M\to\mathbb{R}$ is the smooth map defined by $\hat{\mu}(X)(x)=\mu(x)(X)$ and
$$
\widetilde{X}\left( x\right) =\frac{d}{dt}e^{tX}\cdot x\ _{\left\vert
	t=0\right. }\qquad x\in M,
$$
is the fundamental vector field associated to $X$. Identity (\ref{moment1}) implies that $\widetilde{X}$ is the Hamiltonian vector field of $\hat{\mu}(X)$. If $\Ad^\ast:G\to\textnormal{GL}(\mathfrak{g}^\ast)$ denotes the coadjoint representation of $G$, a moment map $\mu: M\to \mathfrak{g}^\ast$ is called
{\it  $\Ad^\ast$-equivariant} if
$$
\mu \left( g\cdot x\right) =\Ad^{\ast}\left( g\right) \mu \left( x\right)
\qquad g\in G,~x\in M.
$$%
\begin{remark}
	If $\phi_g:M\to M$ is defined by $\phi_g(x)=g\cdot x$ for all $x\in M$, since $\cdot:G\times M\to M$ is a symplectic action, then $\phi_{e^{tX}}^\ast\omega=\omega$, or equivalently, $\mathcal{L}_{\widetilde{X}}\omega=0$ for all $X\in\mathfrak{g}$. Therefore, $\widetilde{X}$ is locally Hamiltonian but
	not necessarily globally Hamiltonian. This is the reason why not every symplectic action is a Hamiltonian action. The latter happens if for instance $\textnormal{H}_{dR}^1(M,\mathbb{R})=0$. On the other hand, if the symplectic form is an exact form of a $G$-invariant 1-form, or else if $G$ is connected and semisimple (only if $\textnormal{H}^1(\mathfrak{g},\mathbb{R})=\textnormal{H}^2(\mathfrak{g},\mathbb{R})=0$), then the symplectic action has an $\Ad^\ast$-equivariant moment map, see \cite{Wallach} or \cite{smgrplie}.
\end{remark}
Recall that a submanifold $\iota: \mathcal{L}\hookrightarrow M$ of a symplectic manifold $(M,\omega)$ is called {\it isotropic} if $\iota^\ast \omega =0$. If moreover $\dim(\mathcal{L})=1/2\dim(M)$ then $\mathcal{L}$ is called {\it Lagrangian}. From now on, we assume that $\cdot: G\times M \to M$ is a Hamiltonian action for a connected symplectic manifold $(M,\omega)$ which admits an $\Ad^\ast$-equivariant moment map $\mu$. Let $L$ be a Lie subgroup of $G$ with Lie algebra $\mathfrak{l}$. The problem considered here is to describe those  orbits $Lx$ ($x\in M$) of  $L$ that are Lagrangian submanifolds of  $M$, or more generally isotropic. The following arguments use the moment map  $\mu $ to give necessary and sufficient conditions for the orbit $Lx$ to be isotropic.\\
If $X,Y\in\mathfrak{g}$, it is well known that the Poisson bracket of $\hat{\mu}(X)$ and $\hat{\mu}(Y)$ is given by
$$
\left\lbrace \hat{\mu}(X),\hat{\mu}(Y) \right\rbrace=\omega \left(\widetilde{X},\widetilde{Y}\right)
=-\widetilde{X}\cdot\hat{\mu}(Y)=\widetilde{Y}\cdot\hat{\mu}(X).
$$%
Therefore, for all $x\in M$
\begin{eqnarray*}
	\widetilde{X}\cdot\hat{\mu}(Y)\left( x\right) &=&\frac{d}{dt}\hat{\mu}(Y)\left( e^{tX}x\right)
	_{\left\vert t=0\right. }=\frac{d}{dt}\mu \left( e^{tX}x\right) _{\left\vert
		t=0\right. }\left( Y\right) \\
	&=&\frac{d}{dt}\Ad^\ast\left( e^{tX}\right) \mu \left( x\right)
	_{\left\vert t=0\right. }\left( Y\right) =\left( \ad^\ast\left( X\right)\mu \left( x\right) \right) \left( Y\right) \\
	&=&-\mu \left( x\right) \left( \left[ X,Y\right] \right) .
\end{eqnarray*}%
The above computation implies two things. The first one is that $\hat{\mu}$ defines a Lie algebra homomorphism between $\mathfrak{g}$ and $C^\infty (M)$ seen as Lie algebra with the Poisson bracket. The second one is that 
\begin{equation}\label{anula}\omega_x \left( \widetilde{X}\left( x\right) ,\widetilde{Y}%
\left( x\right) \right) =0 \qquad \textnormal{if and only if} \qquad \mu \left( x\right) \left[ Y,X%
\right] =0.\end{equation}
Recall also that the tangent space to the orbit $Lx$ at the point $x$
is given by
$$
T_{x}\left( Lx\right) =\{\widetilde{X}\left( x\right) :X\in \mathfrak{l}\}.
$$%
Therefore, $\omega _{x}$ vanishes identically on  $T_{x}\left( Lx\right) $
if and only if  $\mu \left( x\right) \left[ X,Y\right]=0$ for all $X,Y\in 
\mathfrak{l}$. Hence, we obtain the following  characterization of those
orbits of  $L$ that are isotropic.

\begin{proposition}\label{anill}
	An orbit $Lx$ is isotropic if and only if  $\mu \left(x\right) 
	$ belongs to the annihilator $\left( \mathfrak{l}^{\prime }\right) ^{\circ }$ of the 
	derived algebra $\mathfrak{l}^{\prime }$ of $\mathfrak{l}$.
\end{proposition}
\begin{proof}
	Choose $y\in Lx$. By observation  (\ref{anula}), the 
	tangent space $T_{y}\left(
	Lx\right) =T_{y}\left( Ly\right) $ is isotropic if and only if $\mu
	\left( y\right) \left[X,Y\right] =0$ for all $X,Y\in \mathfrak{l}$, that 
	is, if and only if $\mu \left( y\right) $ belongs to the annihilator of  $%
	\mathfrak{l}^{\prime }$. But, $\mu \left( y\right) \in \left( \mathfrak{l}%
	^{\prime }\right) ^{\circ }$ if and only if  $\mu \left( x\right) \in \left( 
	\mathfrak{l}^{\prime }\right) ^{\circ }$ since if  $y=gx$ ($g\in L$) then 
	$\mu \left( y\right) =\Ad\left( g\right) ^{\ast }\mu \left( x\right) 
	$ and therefore $\mu \left( y\right) $ annihilates $\mathfrak{l}^{\prime }$ 
	if and only if  $\mu \left( x\right) \left( \Ad\left( g\right) \mathfrak{l}%
	^{\prime }\right) =\mu \left( x\right) \left( \mathfrak{l}^{\prime }\right)
	=0$ given that $\mathfrak{l}^{\prime }$ is  invariant by every automorphism of  $\mathfrak{l}$.
\end{proof}

\begin{remark}
	\begin{itemize}
		\item It is worth noticing that  the criterion given in  Proposition \ref{anill}
		needs to be verified only at a single point of the orbit  $Lx$ given that
		the annihilator  $\left( \mathfrak{l}^{\prime }\right) ^{\circ }$ is 
		invariant by the coadjoint action.
		\item Proposition \ref{anill} can also be applied to the case $L=G$, although, 
		the reasoning can be carried out  for a pair of groups $L\subset G$. 
		If the action of  $G$ is  Hamiltonian the same is true for the action of $L$ 
		and the moment map $\mu _{L}$, as well as for the 
		moment map  $\mu $ for the restriction of the action of  $G$,
		that is, $\mu _{L}\left( x\right) =\mu \left( x\right) _{\left\vert 
			\mathfrak{l}\right. }$ directly by definition. (The examples 
		we will consider suggest to use a  pair of groups
		$L\subset G$ and to take  $M$ as a homogeneous space of $G$.)
		\item In the  particular case when $(G,\langle\cdot,\cdot\rangle)$ is an orthogonal Lie group (see Definition \ref{ort}) the moment map can be interpreted as a map with values in  $\mathfrak{g}$. In such case $\left( \mathfrak{l}^{\prime }\right) ^{\circ }$ becomes	the orthogonal complement of  $\mathfrak{l}^{\prime}$ with respect to the invariant scalar product $\langle\cdot,\cdot\rangle_e$ induced over $\mathfrak{g}$. Accordingly, an orbit $Lx$ is isotropic if and only if $\mu \left(x\right) $ belongs to the orthogonal complement of $\mathfrak{l}^{\prime }$. In particular, this is the case when $G$ is a compact Lie group where there exists an
		invariant inner product on $\mathfrak{g}$ and when $G$ is a semisimple Lie group replacing the inner product by the Cartan--Killing form.
	\end{itemize}
\end{remark}

\section{Orthogonal Lie groups}\label{ortogonal}
Let $G$ be a connected Lie group with Lie algebra $\mathfrak{g}$ and $\mathfrak{g}^\ast$ its dual vector space. The cotangent bundle of $G$ is isomorphic to the trivial vector bundle $G\times \mathfrak{g}^\ast$ through the isomorphism of vector bundles $\lambda: T^\ast (G)\to G\times \mathfrak{g}^\ast$ defined by
\begin{equation}\label{BodyCoordinates}
\lambda(g,\alpha_g)=(g,\alpha_g\circ (\textnormal{d}L_g)_e)\qquad (g,\alpha_g)\in T^\ast(G). 
\end{equation}
By means of left and right multiplications on $G$ we can define two natural left actions of $G$ on itself, called left and right action, which are given by $L:G\times G\to G$ and $R:G\times G \to G$ and are defined by $L(g,h)=L_g(h)=gh$ and $R(g,h)=R_{g^{-1}}(h)=hg^{-1}$ respectively. These actions can be lifted to $T^\ast(G)$ allowing us to define two left action of $G$ on $T^\ast(G)$ as follows.
$$	\widetilde{L}:G\times T^\ast(G)\to T^\ast(G),\qquad \widetilde{L}(g,(h,\alpha_h))=(gh,\alpha_h\circ (\textnormal{d} L_{g^{-1}})_{gh})\quad\textnormal{and}
$$
$$
\widetilde{R}:G\times T^\ast(G)\to T^\ast(G),\qquad \widetilde{R}(g,(h,\alpha_h))=(hg^{-1},\alpha_h\circ (\textnormal{d} R_g)_{hg^{-1}}).
$$
On the coordinates defined by the formula (\ref{BodyCoordinates}) these left actions are expressed as
\begin{equation}\label{LeftAction}
\widetilde{L}_\lambda(g,(h,\alpha))=(\lambda\circ \widetilde{L}_g\circ\lambda^{-1})(h,\alpha)=(gh,\alpha)\quad\textnormal{and}
\end{equation}
\begin{equation}\label{RightAction}
\widetilde{R}_\lambda(g,(h,\alpha))=(\lambda\circ \widetilde{R}_g\circ\lambda^{-1})(h,\alpha)=(hg^{-1},\Ad^\ast(g)(\alpha)).
\end{equation}
If $\theta_0$ denotes the Liouville 1-form of $T^\ast(G)$, then on the coordinates (\ref{BodyCoordinates}) it is given by $\theta=(\lambda^{-1})^\ast \theta_0$ as
$$
\theta_{(g,\alpha)}(v(g),\beta)=\alpha((\textnormal{d}L_{g^{-1}})_g (v(g))),
$$
and the canonical symplectic form $\omega_0$ of $T^\ast(G)$ is $\omega=(\lambda^{-1})^\ast \omega_0=-\textnormal{d}\theta$ which is given explicitly as
\begin{eqnarray*}
	\omega_{(g,\alpha)}((v(g),\beta),(u(g),\gamma)) & = & \gamma((\textnormal{d}L_{g^{-1}})_g (v(g)))-\beta((\textnormal{d}L_{g^{-1}})_g (u(g)))\\
	& + & \alpha([(\textnormal{d}L_{g^{-1}})_g (v(g)),(\textnormal{d}L_{g^{-1}})_g (u(g))]),
\end{eqnarray*}
where $(g,\alpha)\in G\times \mathfrak{g}^\ast$ and $(v(g),\beta),\ (u(g),\gamma)\in T_{(g,\alpha)}(G\times\mathfrak{g}^\ast)\approx T_g G\times\mathfrak{g}^\ast$.\\
In these terms it is simple to check that (\ref{LeftAction}) and (\ref{RightAction}) are symplectic actions of $G$ over $T^\ast(G)\approx G\times \mathfrak{g}^\ast$. Therefore, as $\omega_0$ (and $\omega$) is an exact symplectic form we have the following $\Ad^\ast$-equivariant moment maps, see \cite{Abraham-Marsden}. For the left action
$$
\mu_L: T^\ast (G)\to\mathfrak{g}^\ast,\qquad \mu_L(g,\alpha_g)=\alpha_g\circ (\textnormal{d}R_g)_e,
$$
which on the coordinates of $G\times \mathfrak{g}^\ast$ is given by
$$
\mu_L^\lambda: G\times \mathfrak{g}^\ast\to\mathfrak{g}^\ast,\qquad \mu_L^\lambda(g,\alpha)=(\mu_L\circ \lambda^{-1})(\alpha)=\Ad^\ast(g)(\alpha).
$$
Analogously, for the right action
$$
\mu_R: T^\ast (G)\to\mathfrak{g}^\ast,\qquad \mu_R(g,\alpha_g)=-\alpha_g\circ (\textnormal{d}L_g)_e,
$$
and also
$$
\mu_R^\lambda: G\times \mathfrak{g}^\ast\to\mathfrak{g}^\ast,\qquad \mu_R^\lambda(g,\alpha)=(\mu_L\circ \lambda^{-1})(\alpha)=-\alpha.
$$
Let $(G,\langle\cdot,\cdot\rangle)$ be an orthogonal Lie group and $(\mathfrak{g},\langle\cdot,\cdot\rangle_e)$ is its respective orthogonal Lie algebra. If $\mathfrak{z}(\mathfrak{g})$ denotes the center of $\mathfrak{g}$ then we obtain
\begin{proposition}\label{anillCotangent}
	Let $(g,\alpha)$ be an element of $G\times \mathfrak{g}^\ast$. Then the orbit $G(g,\alpha)$ by the left action of $G$ on $G\times \mathfrak{g}^\ast$ is isotropic if and only if $\Ad(g)(X_\alpha)\in\mathfrak{z}(\mathfrak{g})$. On the other hand, the orbit $G(g,\alpha)$ by the right action of $G$ on $G\times \mathfrak{g}^\ast$ is isotropic if and only if $X_\alpha\in\mathfrak{z}(\mathfrak{g})$.
\end{proposition}
\begin{proof}
	Let $(g,\alpha)$ be an element of $G\times \mathfrak{g}^\ast$ and $X_\alpha$ the unique element of $\mathfrak{g}$ such that $\alpha(\cdot)=\langle X_\alpha,\cdot\rangle_e$. As $\langle\cdot,\cdot\rangle$ is a bi-invariant pseudo-metric over $G$, then $\langle \Ad(g)(X),\Ad(g)(Y)\rangle_e=\langle X,Y\rangle_e$ for all $g\in G$ and $X,Y\in \mathfrak{g}$. As $\langle\cdot,\cdot\rangle_e$ is an invariant scalar product on $\mathfrak{g}$, by Proposition \ref{anill} the orbit $G(g,\alpha)$ by the left action of $G$ on $G\times \mathfrak{g}^\ast$ is isotropic if and only if $\mu_L^\lambda(g,\alpha)([X,Y])=0$ for all $X,Y\in\mathfrak{g}$, but
	\begin{eqnarray*}
		\mu_L^\lambda(g,\alpha)([X,Y]) & = &  \Ad^\ast(g)(\alpha)([X,Y])=\alpha(\Ad_{g^{-1}}([X,Y]))\\
		& = & \langle X_\alpha,\Ad(g^{-1})([X,Y])\rangle_e = \langle \Ad(g)(X_\alpha),[X,Y]\rangle_e\\
		& = & \langle [\Ad(g)(X_\alpha),X],Y\rangle_e.
	\end{eqnarray*}
	Therefore, the orbit $G(g,\alpha)$ is isotropic if and only if $\langle [\Ad(g)(X_\alpha),X],Y\rangle_e=0$ for all $X,Y\in\mathfrak{g}$. As $\langle\cdot,\cdot\rangle_e$ is nondegenerate the above happens if and only if $[\Ad(g)(X_\alpha),X]=0$ for all $X\in\mathfrak{g}$, that is, $\Ad(g)(X_\alpha)\in\mathfrak{z}(\mathfrak{g})$.\\
	By a similar way, for the right action we have that
	$$
	\mu_R^\lambda(g,\alpha)([X,Y])=-\alpha([X,Y])=- \langle X_\alpha,[X,Y]\rangle_e=-\langle [X_\alpha,X],Y\rangle_e.
	$$
	Thus, the orbit $G(g,\alpha)$ by the right action of $G$ on $G\times \mathfrak{g}^\ast$ is isotropic if and only if $\langle [X_\alpha,X],Y\rangle_e=0$ for all $X,Y\in\mathfrak{g}$, that is, $X_\alpha\in \mathfrak{z}(\mathfrak{g})$.
\end{proof}
An immediate consequence of the previous Proposition is the following.

\begin{corollary}\label{SemisimpleCaseCotangent}
	Let $G$ be a semisimple Lie group. The only orbits by the left and right actions of $G$ on $G\times \mathfrak{g}^\ast$ that are isotropic are of the form $G(g,0)$ for all $g\in G$. Such orbits are Lagrangian and Hamiltonian isotopic to $G$.
\end{corollary}
\begin{proof}
	As $G$ is a semisimple Lie group, then $\mathfrak{g}$ is a semisimple Lie algebra. Therefore $\mathfrak{z}(\mathfrak{g})=0$. Thus, as $\Ad(g):\mathfrak{g}\to\mathfrak{g}$ is a linear isomorphism, the result  follows immediately.
\end{proof}

\section{Complex flag manifolds\label{secexflagcplx}}
Let $U$ be a compact semisimple Lie group with Lie algebra $\mathfrak{u}$. The adjoint orbits of $U$ in  $\mathfrak{u}$ are the  flags of the 
complex group $G$ which has the Lie algebra  $\mathfrak{g}=\mathfrak{u}_{%
	\mathbb{C}}$. If $\langle\cdot,\cdot\rangle$ denotes the Cartan--Killing form on $\mathfrak{u}$, the Kirillov--Kostant--Souriau (KKS) symplectic form 
is given by  $$\omega _{x}\left( \widetilde{X}\left( x\right) ,\widetilde{Y}%
\left( x\right) \right) =\langle x,\left[ X,Y\right] \rangle $$ and the  Hamitonian vector field  $%
\widetilde{X}=\ad\left( X\right) $ corresponds to the
Hamiltonian function  $\hat{\mu}(X)\left( x\right) =\langle x,X\rangle $. Therefore,
the moment map  $\mu $ is the identity map,
which is (evidently) equivariant (for more details see Appendix \ref{OrthogonalLieGroups}).

\begin{remark} Adjoint orbits are the only homogeneous spaces	of a compact semi-simple group  $U$ for which 
	the action of  $U$ is  Hamiltonian. It is so because if  the action of $U$
	on a manifold $M$ is transitive and  Hamiltonian, then the  moment map 
	$\mu $ is a covering  space over the adjoint orbit (the complex flag).
	Since complex flags are simply connected, we conclude that  $M$ is 
	itself the adjoint orbit, see \cite[p.\thinspace 341]{smgrplie}.
\end{remark}

Let $\mathfrak{t}=i\mathfrak{h}_{\mathbb{R}}$ be the  Lie algebra of a 
maximal torus $T=\langle \textnormal{exp}\mathfrak{t}\rangle=\textnormal{exp}\mathfrak{t}$ in $U$. Here $\mathfrak{h}_\mathbb{R}$ is defined as follows. As the restriction of the Cartan-Killing form to $\mathfrak{t}$ is nondegenerate, if $\alpha$ is a root of $\mathfrak{t}$, there exists a unique $H_\alpha\in\mathfrak{t}$ such that $\alpha(\cdot)=\langle H_\alpha,\cdot\rangle$. The real subspace generated by $H_\alpha$, with $\alpha$ root of $\mathfrak{t}$, is denoted by $\mathfrak{h}_\mathbb{R}$. Then every adjoint orbit has the form  $\Ad%
\left( U\right) \left( iH_{0}\right) $ with  $H_{0}\in \mathfrak{h}_{\mathbb{R}%
}$, which may be chosen in the closure of the positive Weyl chamber. According to Section \ref{isotropic},
a  subgroup $L\subset U$ with  Lie algebra $%
\mathfrak{l}$ admits an isotropic orbit in  $\Ad\left(
U\right) \left( iH_{0}\right) $ if and only if  $\Ad\left( U\right)
\left( iH_{0}\right) \cap \left( l^{\prime }\right) ^{\bot }\neq \emptyset $,
that is, if there exists  $u\in U$ such that $\Ad\left( u\right) \left(
iH_{0}\right) \in \left( l^{\prime }\right) ^{\bot }$. In such a case, the isotropic
$L$-orbits are the orbits of  $X\in \Ad%
\left( U\right) \left( iH_{0}\right) \cap \left( l^{\prime }\right) ^{\bot }$. Focusing on the subgroup instead, these observations can be reinterpreted as follows. If $L\subset U$ is a proper subgroup, then $L$ has an 
isotropic orbit in some flag $\Ad\left( U\right) \left( iH_{0}\right) $ with $H_{0}\neq 0$. The reason is that
if  $L$ is proper then  $\left( l^{\prime }\right) ^{\bot }\neq \{0\}$ and
if $0\neq X\in \left( l^{\prime }\right) ^{\bot }$ then $X$ belongs to 
some orbit $\Ad\left( U\right) \left( iH_{0}\right) $ with $%
H_{0}\neq 0$. A particular case is the orbit of  $L$ through the origin  $x_{0}=iH_{0}$. It 
is isotropic if and only if  $l^{\prime }\subset \left( iH_{0}\right)
^{\bot }$. A sufficient condition is that $l^{\prime
}\subset \mathfrak{t}^{\bot }$.

\begin{example}\label{su}
	Choose $U=\mathrm{SU}\left( n\right) $ and $L=\mathrm{SO}\left( n\right) $.
	Take $\mathfrak{t}\subset \mathfrak{su}\left( n\right) $ as the 
	subalgebra of diagonal matrices ($\mathrm{tr}=0$). A matrix in  $%
	\mathfrak{so}\left( n\right) $ (real anti-symmetric matrix) has
	zeros in the  diagonal which implies that $\mathfrak{so}\left( n\right)
	\subset \mathfrak{t}^{\bot }$. Thus, the orbit of  $\mathrm{SO}%
	\left( n\right) $ through the origin of any  flag is isotropic. Since
	these orbits have half the dimension of the respective  flags they are in fact
	Lagrangian orbits.
	This example gives an instance of the well-known construction  of the immersion of real 
	flags into complex flags.
\end{example}

\begin{example}\label{Gsu}
	Example \ref{su} may be generalized as follows: choose a
	Weyl basis of  $\mathfrak{g}=\mathfrak{u}_{\mathbb{C}}$ containing $X_{\alpha }\in 
	\mathfrak{g}_{\alpha }$ with $\left[ X_{\alpha },X_{\beta }\right] =m_{\alpha
		,\beta }X_{\alpha +\beta }$. For each root $\alpha $ define $A_{\alpha
	}=X_{\alpha }-X_{-\alpha }$. Then, $A_{\alpha }\in \mathfrak{u}$ (by the 
	canonical construction of  $\mathfrak{u}$) and 
	$$
	{[}A_{\alpha },A_{\beta }{]}=m_{\alpha ,\beta }A_{\alpha +\beta }+m_{-\alpha
		,\beta }A_{\alpha -\beta }.
	$$%
	Consequently, the subspace $\mathfrak{l}$ generated by the $A_{\alpha }$'s 
	is a  subalgebra perpendicular to the  Cartan subalgebra.
	Therefore, for any  Lie group  $L$ with Lie algebra $\mathfrak{l}
	$, the orbits $Lx_{0}$ through the origin of the  flags are isotropic 
	submanifolds. Actually, these orbits are Lagrangians since they have half the dimension
	of the flags.
\end{example}

\begin{example} Choose  $H\in \mathfrak{t}$ and let $U_{H}=\{u\in U:\Ad\left(
	u\right) H =H \}$
	be the centralizer of  $H$ in $U$. Its Lie algebra is  
	$$\mathfrak{u}_{H}= \{X\in \mathfrak{u}:\left[ H,X\right] =0\} = \mathfrak{t}\oplus \sum_{\alpha \left( H\right) =0}\mathfrak{u}_{\alpha },$$
	where $\mathfrak{u}_{\alpha }=\mathfrak{u}\cap \left( \mathfrak{g}_{\alpha
	}\oplus \mathfrak{g}_{-\alpha }\right) $. For example, if $H$ is regular,
	then $\mathfrak{u}_{H}=\mathfrak{t}$ and since $\mathfrak{t}$ is
	abelian, $\mathfrak{t}'=\{0\}$, and accordingly  every flag (adjoint orbit)
	intersects $\left( \mathfrak{t}^{\prime }\right) ^{\bot }=\mathfrak{u}$ and
	therefore $U_{H}$ has isotropic orbits in all  flags. In fact, all these
	orbits of  $T$ are isotropic since $\mathfrak{t%
	}$ is Abelian.
\end{example}

\begin{example} \label{example-su3}To obtain a more interesting example, take $\mathfrak{u}=%
	\mathfrak{su}\left( 3\right) $ and
	$$
	H=i\mathrm{diag}\{2,-1,-1\}.
	$$%
	Then $\mathfrak{u}_{H}$ is given by matrices of the form
	\begin{equation}\label{diag}
	\left( 
	\begin{array}{cc}
	it & 0 \\ 
	0 & A%
	\end{array}%
	\right)
	\end{equation}%
	with $A\in \mathfrak{u}\left( 2\right) $ and $it+\mathrm{tr}A=0$. The derived
	algebra $\mathfrak{u}_{H}^{\prime }$ is given by the matrices in  $%
	\mathfrak{u}_{H}$ such that $t=0$, $A\in \mathfrak{su}\left( 2\right) $ and $%
	\dim _{\mathbb{R}}\mathfrak{u}_{H}^{\prime }=3$. The orthogonal complement $%
	\left( \mathfrak{u}_{H}^{\prime }\right) ^{\bot }$ is the   $5$-dimensional space of matrices 
	\begin{equation}
	X=\left( 
	\begin{array}{ccc}
	2it & z & w \\ 
	-\overline{z} & -it & 0 \\ 
	-\overline{w} & 0 & -it%
	\end{array}%
	\right) \qquad t\in \mathbb{R},z,w\in \mathbb{C}.
	\label{formatrizortogonalex1}
	\end{equation}%
	This orthogonal complement intersects the 3 types of adjoint 
	orbits which give the flags $\mathbb{CP}^{2}$, $\mathrm{Gr}_{2}\left( 3,%
	\mathbb{C}\right) $ and  $\mathbb{F}\left( 1,2\right) $. In the case of  $\mathbb{F}%
	\left( 1,2\right) $ it is easy to find a matrix in   $\left( 
	\mathfrak{u}_{H}^{\prime }\right) ^{\bot }$ which is  regular. For example,
	the matrix
	$$\left( 
	\begin{array}{ccc}
	0 & 1 & 0 \\ 
	-1 & 0 & 0 \\ 
	0 & 0 & 0%
	\end{array}%
	\right)
	\in \left( \mathfrak{u}_{H}^{\prime }\right) ^{\bot }
	$$
	which has eigenvalues $i,-i,0$. This matrix is regular and its adjoint orbit is the maximal flag. For the cases of partial flags the calculations require a further details.
	
	\vspace{3mm}
	\noindent {\sc (i) The case of  $\mathbb{CP}^{2}$}. Here we need to find a matrix in  $\left( 
	\mathfrak{u}_{H}^{\prime }\right) ^{\bot }$ which is  conjugate to a matrix of 
	the type  $D=\mathrm{diag}\{i2a,-ia,-ia\}$ with $a>0$. Any matrix   $X$ in (\ref%
	{formatrizortogonalex1}) has trace $0$  and its characteristic polynomial
	has the form $\lambda ^{3}+F\lambda +G$ with $F=3t^{2}+\left\vert
	z\right\vert ^{2}+\left\vert w\right\vert ^{2}$ and $G=-\det X=it\left(
	2t^{2}+\left\vert w\right\vert ^{2}+\left\vert z\right\vert ^{2}\right) $. 
	The characteristic polynomial of  $D$ is $\lambda ^{3}+3a^{2}\lambda +2ia^{3}$. The two matrices are conjugate  if and only if their characteristic 
	polynomials coincide since both are anti-Hermitian. This happens
	if and only if 
	\begin{eqnarray*}
		2a^{3} &=&t\left( 2t^{2}+\left\vert z\right\vert ^{2}+\left\vert
		w\right\vert ^{2}\right) \\
		3a^{2} &=&3t^{2}+\left\vert z\right\vert ^{2}+\left\vert w\right\vert ^{2}.
	\end{eqnarray*}%
	If $t=0$, then $\det X=0$ and the eigenvalues of  $X$ are $0$ and $\pm i%
	\sqrt{A}$ (since $A>0$). In such case, $X$ is  regular and not
	conjugate to $D$. On the other hand, we can multiply both matrices by 
	a positive constant without affecting conjugation. Thus, it is 
	enough to verify the existence of solutions for  $t=\pm 1$. Setting $t=1$, the previous equations become
	\begin{eqnarray*}
		2a^{3} &=&2+\left\vert z\right\vert ^{2}+\left\vert w\right\vert ^{2} \\
		3a^{2} &=&3+\left\vert z\right\vert ^{2}+\left\vert w\right\vert ^{2}
	\end{eqnarray*}%
	which is equivalent to $2a^{3}-2=3a^{2}-3=\left\vert w\right\vert
	^{2}+\left\vert z\right\vert ^{2}$. The first equality gives the polynomial
	$2a^{3}-3a^{2}+1=0$
	which has roots  $1$ with multiplicity $2$ and $-\frac{1}{2}$. Thus,
	we must have $a=1$. However, if  $a=1$ then $2a^{3}-2=0$ which
	implies that $z=w=0$. Hence the only possible solution is obtained by taking $X=%
	\mathrm{diag}\{2i,-i,-i\}$. On the other hand, if  $t=-1$ then the equations become
	\begin{eqnarray*}
		2a^{3} &=&-2-\left\vert z\right\vert ^{2}-\left\vert w\right\vert ^{2} \\
		3a^{2} &=&3+\left\vert z\right\vert ^{2}+\left\vert w\right\vert ^{2}
	\end{eqnarray*}%
	that is , $-2a^{3}-2=3a^{2}-3=\left\vert w\right\vert ^{2}+\left\vert
	z\right\vert ^{2}$. The first equality gives the polynomial $%
	2a^{3}+3a^{2}-1=0$, which has roots $-1$ with multiplicity $2$ and $%
	\frac{1}{2}$. The only possible solution here is obtained by taking $a=1/2$ which
	would give the contradiction $\left\vert w\right\vert ^{2}+\left\vert z\right\vert
	^{2}=-9/4$. Therefore a conjugation does not exist. Summing up, the only possible  isotropic orbit in the projective plane is 
	the trivial one (reduced to a point) passing through $H=\mathrm{diag}\{2i,-i,-i\}$.
	\vspace{3mm}
	
	\noindent{\sc (ii) The case of  $\mathrm{Gr}_{2}\left( 3,\mathbb{C}\right) $}. We need to find a 
	matrix in  (\ref{formatrizortogonalex1}) conjugate to $\mathrm{diag}%
	\{ia,ia,-2ia\}$ with $a>0$. The equations are similar to the ones just considered. 
	The characteristic polynomial now becomes $\lambda ^{3}+3a^{2}\lambda -2ia^{3}$ and
	the equations required for existence of a conjugation are
	\begin{eqnarray*}
		-2a^{3} &=&t\left( 2t^{2}+\left\vert z\right\vert ^{2}+\left\vert
		w\right\vert ^{2}\right) \\
		3a^{2} &=&3t^{2}+\left\vert z\right\vert ^{2}+\left\vert w\right\vert ^{2}.
	\end{eqnarray*}%
	Once again, it is enough to verify the cases when $t=\pm 1$. For $t=1$ we have $-2a^{3}-2=3a^{2}-3=\left\vert z\right\vert ^{2}+\left\vert
	w\right\vert ^{2}$. The first equality is the same one found in the second case 
	of the previous example, and gives no solution. On the other hand, for $t=-1$ we have $-2a^{3}+2=3a^{2}-3=\left\vert z\right\vert ^{2}+\left\vert
	w\right\vert ^{2}$. The polynomial is  $2a^{3}+3a^{2}-5=0$ whose
	roots are $1$ and $-\frac{5}{4}\pm \frac{1}{4}i\sqrt{15}$. Thus, we must have
	$a=1$ and consequently $\left\vert z\right\vert ^{2}+\left\vert
	w\right\vert ^{2}=0$, that is, $z=w=0$. Therefore the only solution 
	is the orbit of  $H_{1}=\mathrm{diag}\{i,i,-2i\}$.\\
	Opposite to the case of  $\mathbb{C}P^{2}$, this isotropic orbit in  $%
	\mathrm{Gr}_{2}\left( 3,\mathbb{C}\right) $ is not trivial. In fact,
	a generic matrix in  $\mathfrak{u}_{H}$ has the form
	$$
	Y=\left( 
	\begin{array}{ccc}
	2it & 0 & 0 \\ 
	0 & -it & z \\ 
	0 & -\overline{z} & -it%
	\end{array}%
	\right) \qquad t\in \mathbb{R},z\in \mathbb{C}.
	$$%
	If $H$ is the  diagonal matrix in  (\ref{diag}), then
	$$
	\ad\left( Y\right) H=-\left[ H,Y\right] =\left( 
	\begin{array}{ccc}
	0 & 0 & 0 \\ 
	0 & 0 & -3iz \\ 
	0 & 3i\overline{z} & 0%
	\end{array}%
	\right) .
	$$%
	The  tangent space to the orbit is  generated by the latter matrices
	with $z$ varying in $\mathbb{C}$. Hence, the orbit has
	real dimension $2$. Since $\dim _{\mathbb{R}}\mathrm{Gr}_{2}\left( 3,
	\mathbb{C}\right) =4$, this isotropic orbit is in fact
	Lagrangian.
	\vspace{3mm}
	
	\noindent{\sc (iii) The case of $\mathbb{F}\left( 1,2\right) $}.
	To conclude this example we need to analyze the isotropic  
	orbits of  $U_{H}$ in the  maximal flag $\mathbb{F}=\mathbb{F}\left( 1,2\right) $. To do so we need to choose a realization of   $\mathbb{F}$
	as adjoint orbit $\Ad\left( U\right) \left( iH_{0}\right) $ and
	to analyze the orbits of the action of  $U_{H}$ in the intersection
	of this orbit with the orthogonal  complement $\left( \mathfrak{u}%
	_{H}^{\prime }\right) ^{\bot }$, which is the  $5$ dimensional space formed 
	by the matrices
	\begin{equation}\label{perpmatrix}
	X=\left( 
	\begin{array}{ccc}
	2it & z & w \\ 
	-\overline{z} & -it & 0 \\ 
	-\overline{w} & 0 & -it%
	\end{array}%
	\right) \qquad t\in \mathbb{R},z,w\in \mathbb{C}.
	\end{equation}%
	Choose $iH_{0}=\mathrm{diag}\{i,0,-i\}$, which is a  regular element.
	Then, the intersection $\left( \mathfrak{u}_{H}^{\prime }\right) ^{\bot
	}\cap \Ad\left( U\right) \left( iH_{0}\right) $ is formed by  
	matrices of the form (\ref{perpmatrix}), 
	whose eigenvalues are $0$ and $\pm i$, since 
	\textrm{A}$\mathrm{d}\left( U\right) \left( iH_{0}\right) $ is the set
	of matrices in  $\mathfrak{su}\left( 3\right) $ which have the same 
	eigenvalues as $iH_{0}$. Since $0$ is an eigenvalue of $X\in \left( 
	\mathfrak{u}_{H}^{\prime }\right) ^{\bot }\cap \Ad\left( U\right)
	\left( iH_{0}\right) $, we must have $\det X=-it\left( 2t^{2}+\left\vert
	w\right\vert ^{2}+\left\vert z\right\vert ^{2}\right) $ which happens if and only if
	$t=0$, given that $t\in \mathbb{R}$. Thus, the characteristic polynomial of 
	$X$ becomes $\lambda ^{3}+\left( \left\vert
	z\right\vert ^{2}+\left\vert w\right\vert ^{2}\right) \lambda $ implying 
	that $\left\vert z\right\vert ^{2}+\left\vert w\right\vert ^{2}=1$.
	This shows that $\left( \mathfrak{u}_{H}^{\prime }\right) ^{\bot }\cap 
	\Ad\left( U\right) \left( iH_{0}\right) $ is formed by the 
	matrices satisfying
	$$
	\qquad \left\vert z\right\vert ^{2}+\left\vert w\right\vert ^{2}=1
	$$%
	describing the sphere $S^{3}$ in $\mathbb{C}^{2}$.
	On the other hand, the group $U_{H}$ is 
	isomorphic  to $U\left( 2\right) $ and its action in $\left( \mathfrak{u}%
	_{H}^{\prime }\right) ^{\bot }\cap \Ad\left( U\right) \left(
	iH_{0}\right) $ is the same action of  $U\left( 2\right) $ in $S^{3}$.
	Therefore $U_{H}$ acts transitively in $\left( \mathfrak{u}%
	_{H}^{\prime }\right) ^{\bot }\cap \Ad\left( U\right) \left(
	iH_{0}\right) $ which is the unique isotropic orbit. In fact, 
	Lagrangian since its dimension is $3=\frac{1}{2}\dim _{%
		\mathbb{R}}\mathbb{F}$.\\
	Still in this case,  it is interesting to regard the 
	Lagrangian orbit ($\approx S^{3}$) more intrinsically, in terms of 
	flags of subspaces $\left( V_{1}\subset V_{2}\right) $ with $\dim
	V_{i}=i $. This is done observing that 
	\begin{equation}\label{k}
	k=\frac{\sqrt{2}}{2}\left( 
	\begin{array}{ccc}
	1 & 0 & -1 \\ 
	0 & 1 & 0 \\ 
	1 & 0 & 1%
	\end{array}%
	\right)
	\end{equation}%
	and
	\begin{eqnarray*}
		\Ad\left( k\right) \left( 
		\begin{array}{ccc}
			i & 0 & 0 \\ 
			0 & 0 & 0 \\ 
			0 & 0 & -i%
		\end{array}%
		\right) &=&\frac{1}{2}\left( 
		\begin{array}{ccc}
			1 & 0 & -1 \\ 
			0 & 1 & 0 \\ 
			1 & 0 & 1%
		\end{array}%
		\right) \left( 
		\begin{array}{ccc}
			i & 0 & 0 \\ 
			0 & 0 & 0 \\ 
			0 & 0 & -i%
		\end{array}%
		\right) \left( 
		\begin{array}{ccc}
			1 & 0 & 1 \\ 
			0 & 1 & 0 \\ 
			-1 & 0 & 1%
		\end{array}%
		\right) \\
		&=&\left( 
		\begin{array}{ccc}
			0 & 0 & i \\ 
			0 & 0 & 0 \\ 
			i & 0 & 0%
		\end{array}%
		\right) \in \left( \mathfrak{u}_{H}^{\prime }\right) ^{\bot }\cap \Ad%
		\left( U\right) \left( iH_{0}\right) .
	\end{eqnarray*}
	Since $x_{0}=iH_{0}$ is the origin of the maximal flag, this conjugation
	means that the  Lagrangian orbit of $U_{H}$ is the orbit through $%
	kx_{0}$. In the other hand, looking at $\mathbb{F}\left( 1,2\right) $ as the
	set of  flags $\left( V_{1}\subset V_{2}\right) $ with $\dim V_{i}=i$ 
	the origin is $f_{0}=\left( \langle e_{1}\rangle \subset \langle
	e_{1},e_{2}\rangle \right) $ where $\{e_{1},e_{2},e_{3}\}$ is the  
	canonical basis of  $\mathbb{C}^{3}$. In this representation, the 
	Lagrangian orbit of  $U_{H}$ is the orbit through 
	$$
	kf_{0}=\left( \langle ke_{1}\rangle \subset \langle ke_{1},ke_{2}\rangle
	\right) .
	$$%
	By the expression (\ref{k}) of  $k$ it follows that $ke_{1}=\frac{\sqrt{2}}{2}%
	\left( e_{1}+e_{3}\right) $ and $ke_{2}=e_{2}$ implying that
	$$
	kf_{0}=\left( \langle e_{1}+e_{3}\rangle \subset \langle
	e_{1}+e_{3},e_{2}\rangle \right) .
	$$%
	Here, $U_{H}$ is the embedding of  $\mathrm{U}\left( 2\right) $ in $%
	\mathrm{SU}\left( 3\right) $ given by the matrices%
	$$
	\left( 
	\begin{array}{cc}
	z & 0 \\ 
	0 & g%
	\end{array}%
	\right) \qquad \left\vert z\right\vert =1,~g\in \mathrm{U}\left( 2\right) .
	$$%
	The orbit of this group through $kf_{0}$ is given by 
	$$
	U_{H}kf_{0}=\{\left( \langle e_{1}+ue_{3}\rangle \subset \langle
	e_{1}+ue_{3},ue_{2}\rangle \right) :u\in \mathrm{SU}\left( 2\right) \}
	$$%
	where in this expression $\mathrm{SU}\left( 2\right) $ is seen as the 
	unitary group of the subspace generated by $e_{2},e_{3}$.
\end{example}

\begin{example} Let $Z_H$ be the complexification of  $U_{H}$
	and let  $\mathfrak{z}_{H}$ be the complexification of $\mathfrak{u}_{H}$.
	Then $Z_H$ is the centralizer of $H$ in $G$, and 
	$\mathfrak{z}_{H}$  is the centralizer of $H$ in $\mathfrak{g}$. As observed in \cite{bedori}, if  $U_{H}x$ is 
	a Lagrangian orbit (in a  K\"{a}hler manifold $M$), then the
	orbit $Z_{H}x$ of the complexification is open. The reason is that the
	tangent space $T_{x}U_{H}x$ is Lagrangian and therefore  the subspace $J\left(
	T_{x}U_{H}x\right) $ is the complement of $T_{x}U_{H}x$ in $%
	T_{x}M$ where $J$ is a  complex structure. The tangent space $%
	T_{x}Z_{H}x$ contains $J\left( T_{x}U_{H}x\right) $ since, if  $X\in 
	\mathfrak{u}_{H}$ then $J\left( \widetilde{X}\left( x\right) \right) =%
	\widetilde{iX}\left( x\right) $ and $iX\in \mathfrak{z}_{H}$.
\end{example}
In the previous examples, $U_{H}$ has a  Lagrangian orbit  in the  maximal flag,
and therefore $Z_{H}$ has an open orbit.
\begin{remark} 
	\begin{itemize}
		\item The examples using  $\mathrm{SU}\left( 3\right) $
		above show that in general the projection between 2 flags does not  
		take isotropic orbits to isotropic orbits. In fact, the projection of the Lagrangian 
		orbit in the maximal flag  $\mathbb{F}\left( 1,2\right) $
		to the  projective plane $\mathbb{CP}^{2}$ 
		is not isotropic, give that it does not project on the origin, 
		which is the only isotropic orbit in
		$\mathbb{CP}^{2}$.
		\item If a subgroup $U_{H}$ has a 
		Lagrangian orbit in a  flag $\mathbb{F}_{\Theta }$ 
		then its complexification
		$Z_{H} $ has an open orbit in  $\mathbb{F}_{\Theta }$. The converse is not true,
		since in the previous example $U_{H}$ only has a Lagrangian 
		orbit in the maximal  flag. Therefore $Z_{H}$ has an open orbit in the maximal flag, and so 
		it follows that $Z_{H}$ has an open orbits in every flag.
		\item \noindent{\sc Isotropy representation}: Let $U$ and $L\subset U$ be compact groups. The isotropic orbits of  
		$L$ inside the adjoint orbits of  $U$ are essentially given by 
		the orbits of the isotropy representation of $L$ in $U/L$. For
		example, is $\mathfrak{l}^{\prime }=\mathfrak{l}$ then $\left( 
		\mathfrak{l}^{\prime }\right) ^{\bot }=\mathfrak{l}^{\bot }$, which is 
		identified to the tangent space at the origin of  $U/L$. The adjoint 
		representation of $L$ in $\left( \mathfrak{l}^{\prime }\right) ^{\bot }=%
		\mathfrak{l}^{\bot }$ is the isotropy representation. So that the  
		orbits of  $L$ in $\Ad\left( \mathfrak{l}^{\prime }\right) ^{\bot }=%
		\mathfrak{l}^{\bot }$ which are the isotropic orbits 
		in the adjoint orbits of  $U$ are the same as the orbits
		by the isotropy representation. When $\mathfrak{l}^{\prime }\neq \mathfrak{l}$, the situation does not
		change much. In fact, because of compactness $\mathfrak{l}$ is 
		reductible and  $\mathfrak{l}=\mathfrak{z}_{\mathfrak{l}}\oplus \mathfrak{l}%
		^{\prime }$ where $\mathfrak{z}_{\mathfrak{l}}$ is the centre of  $\mathfrak{l%
		}$. Moreover, $\mathfrak{z}_{\mathfrak{l}}$ is orthogonal to  $%
		\mathfrak{l}^{\prime }$ (with respect to the  Cartan--Killing form of 
		$\mathfrak{u}$). Thus, $\left( \mathfrak{l}^{\prime }\right) ^{\bot }=%
		\mathfrak{z}_{\mathfrak{l}}\oplus \mathfrak{l}^{\bot }$. The adjoint 
		orbits in  $\mathfrak{l}^{\bot }$ are the orbits of the isotropy 
		representation of  $L$ on $U/L$. Now if $X=Z+Y\in \mathfrak{z}_{%
			\mathfrak{l}}\oplus \mathfrak{l}^{\bot }$ then 
		$$
		\Ad\left( L\right) \left( Z+Y\right) =Z+\Ad\left( L\right) Y
		$$%
		which means that the $L$-orbit of $X$ is the  translation by $Y$  of 
		the orbit of  $Y\in \mathfrak{l}^{\bot }$, which is an orbit of the 
		isotropy representation.\\
		Summing up, the orbits of the isotropy representation of  $L$ on $U/L$
		are the same orbits that occur as isotropic orbits in the 
		adjoint orbits of  $U$ (the flags of $G=U_{\mathbb{C}}$).
		\item Suppose that $G$ is a complex semisimple Lie group, not necessarily compact, with $U$ the real compact form of $G$. As the moment map $\mu $ is the identity of $\Ad\left(
		G\right) (H)$, an orbit $Lx$ is isotropic if and only if $x\in
		\left( \mathfrak{l}^{\prime }\right) ^{\bot }$ where the orthogonal is 
		taken with respect to the Cartan--Killing form  $\langle \cdot ,\cdot \rangle $. Some examples of this case are:
		\begin{enumerate}
			\item If $L=U$ then $\left( \mathfrak{u}^{\prime }\right) ^{\bot }=%
			\mathfrak{u}^{\bot }=i\mathfrak{u}=\mathfrak{s}$ (the symmetric 
			part of the  Cartan decomposition $\mathfrak{g}=\mathfrak{u}\oplus i\mathfrak{u}
			$). Thus $U$ has 
			isotropic orbits in  $\Ad\left( G\right) (H)$ if and only if  $%
			H\in \mathfrak{s}$. In such case there exists  a unique isotropic orbit, 
			which is the  flag $\Ad\left( G\right) (H)\cap \mathfrak{s}$.
			\item If $\mathfrak{l}$ is a  Borel subalgebra (minimal
			parabolic) $\mathfrak{p}=\mathfrak{h}\oplus \mathfrak{n}^{+}$ then $%
			\mathfrak{p}^{\prime }=\mathfrak{n}^{+}$ and $\left( \mathfrak{p}^{\prime
			}\right) ^{\bot }$ is the opposite  Borel subalgebra $\mathfrak{p}^{-}=%
			\mathfrak{h}\oplus \mathfrak{n}^{-}$.
		\end{enumerate}
	\end{itemize}
\end{remark}

\section{Products of  flags\label{secdiag}}
Assume here that $U$ is a compact and connected semisimple Lie group with Lie algebra $\mathfrak{u}$. Let $G$ be a complex Lie group with Lie algebra $\mathfrak{g}:=\mathfrak{u}_\mathbb{C}$. The goal here is to describe isotropic or Lagrangian orbits in a product of flags $\mathbb{F}_{\Theta _{1}}\times \mathbb{F}_{\Theta
	_{2}} $ of $G$. A cartesian product $\mathbb{F}_{\Theta
	_{1}}\times \mathbb{F}_{\Theta _{2}}$ of 2 flags of $G$ may be regarded as 
a  flag of $G\times G$ and therefore it may be seen as an adjoint orbit of the 
compact part $U\times U$ as studied in section \ref{secexflagcplx}. The group
$U$ itself (as does  $G$) acts in the product $\mathbb{F}_{\Theta
	_{1}}\times \mathbb{F}_{\Theta _{2}}$ by the   
diagonal action $k\left( x,y\right) =\left( kx,ky\right) $, $%
k\in U$, $\left( x,y\right) \in \mathbb{F}_{\Theta _{1}}\times \mathbb{F}%
_{\Theta _{2}}$,  that is, $U$ is seen as the subgroup $U\times U$
given by the  diagonal $\Delta _{U}=\{\left( u,u\right) :u\in U\}$ and the
diagonal action is the restriction of the action of  $U\times U$ to $%
\Delta _{U}$.\\
Let us now describe the isotropic orbits by the diagonal action. The Lie algebra of  $U\times U$ is $\mathfrak{u}\times \mathfrak{u}$
and the Lie algebra of the  diagonal $\Delta _{U}$ is the  diagonal $\Delta _{%
	\mathfrak{u}}=\{\left( X,X\right) :X\in \mathfrak{u}\}$. Everything is constructed using
Cartesian products:  Cartan subalgebras $\mathfrak{t}\times 
\mathfrak{t}$ and $\mathfrak{h}\times \mathfrak{h}$,  maximal torus $T\times T$,
Weyl chamber $\mathfrak{a}^{+}\times \mathfrak{a}^{+}$ ($\mathfrak{a}=%
\mathfrak{h}_{\mathbb{R}}$) and  the  flag manifolds that are orbits $%
\Ad\left( U\times U\right) \left( iH_{1},iH_{2}\right) $ with  $%
(H_{1},H_{2})\in \mathfrak{a}^{+}\times \mathfrak{a}^{+}$ where
$$\Ad\left( U\times U\right) \left( iH_{1},iH_{2}\right) =\Ad%
\left( U\right) iH_{1}\times \Ad\left( U\right) iH_{2}.$$
An adjoint orbit (product flag) $\Ad\left( U\times U\right)
\left( iH_{1},iH_{2}\right) $ admits an isotropic orbits by the 
diagonal action if and only if  it intercepts the orthogonal $\Delta _{%
	\mathfrak{u}}^{\bot }$ of the derived algebra $\Delta _{\mathfrak{u}%
}^{\prime }$ of $\Delta _{\mathfrak{u}}\approx \mathfrak{u}$ with $\Delta _{%
	\mathfrak{u}}^{\prime }=\Delta _{\mathfrak{u}}$ since $\mathfrak{u}$ is
semisimple. The orthogonal subspace is given by 
$$
\Delta _{\mathfrak{u}}^{\bot }=\{\left( X,-X\right) :X\in \mathfrak{u}\}
$$
since the  Cartan--Killing form of  $\mathfrak{u}\times \mathfrak{u}$ 
is the sum of the forms in each coordinate. Thus, the adjoint 
orbit $\Ad\left( U\times U\right) \left( iH_{1},iH_{2}\right) $
has an isotropic orbit by the diagonal action if and only if there
exist $u_{1},u_{2}\in U$ such that 
\begin{equation}\label{dualcondition}
\Ad\left( u_{1}\right) \left( iH_{1}\right) =-\Ad\left(
u_{2}\right) \left( iH_{2}\right).
\end{equation}
\begin{remark}
	Denote by $U^p:=U\times\cdots\times U$ the product of $U$ $p$-times. If we consider an arbitrary product of flags $\mathbb{F}_{\Theta _{1}}\times\cdots\times \mathbb{F}_{\Theta
		_{p}} $ which can be identified with an orbit $\Ad\left( U^p\right) \left( iH_{1},\cdots, iH_{p}\right)$, then it is easy to show that this 
	has an isotropic orbit by the diagonal action if and only if there
	exist $u_{1},\cdots,u_{p}\in U$ such that $\displaystyle \sum_{j=1}^{p}\Ad(u_j)(iH_j)=0$.
\end{remark}
For the case of two flags, (\ref{dualcondition}) implies that $iH_{2}=\Ad\left( u_{2}^{-1}u_{1}\right) \left(
-iH_{1}\right) $ which means that $iH_{2}$ belongs to the adjoint orbit 
of $-iH_{1}$. This in turn is equivalent to the statement that  the  flags $%
\mathbb{F}_{\Theta _{1}}=\Ad\left( U\right)(iH_{1})$ and $\mathbb{F}%
_{\Theta _{2}}=\Ad\left( U\right) (iH_{2})$ are dual in the sense that
$\Theta _{2}=\sigma \Theta _{1}$ where $\sigma $ is the 
symmetry of the  Dynkin diagram given by  $\sigma =-w_{0}$ and  $w_{0}$ is the
main involution (element of greatest length) of the  Weyl  group
$\mathcal{W}$. In fact, $H_{2}=\iota \left( H_{1}\right) =-w_{0}\left(
H_{1}\right) $ if and only if $-H_{1}$ belongs to the adjoint orbit of $%
H_{2}$. Summing up,
\begin{theorem}\label{noniso}
	A product of flags $\mathbb{F}_{\Theta _{1}}\times \mathbb{F}_{\Theta _{2}}$
	admits an isotropic orbit by the diagonal action if and only if \,
	$\mathbb{F}_{\Theta _{2}}$ is the dual  flag $\mathbb{F}_{\Theta
		_{1}^{\ast }}$ of \, $\mathbb{F}_{\Theta _{1}}$.
\end{theorem}
Assuming that the  flags are dual, that is, $-iH_{1}\in \Ad%
\left( U\right) \left( iH_{2}\right) $, then the isotropic orbits by the 
diagonal action on $\Ad\left( U\times U\right) \left(
iH_{1},iH_{2}\right) $ are those that pass through elements of the type $%
\left( X,-X\right) $ inside the adjoint orbit. Given an element $\left( X,-X\right) \in \Ad\left( U\right) (iH_{1})\times \Ad%
\left( U\right) (iH_{2})$ set $X=\Ad\left( u\right) \left( iH_{1}\right) $ with $u\in U$. Then,
$-X=\Ad\left( u\right) \left( -iH_{1}\right) $, that is, 
$$
\left( X,-X\right) =\left( \Ad\left( u\right) \left( iH_{1}\right) ,%
\Ad\left( u\right) \left( -iH_{1}\right) \right)
$$%
which means that  $\left( X,-X\right) $ belongs to the  diagonal orbit of
$\left( iH_{1},-iH_{1}\right) $ and reciprocally, the elements of the 
diagonal orbit of  $\left( iH_{1},-iH_{1}\right) $ have the form $\left(
X,-X\right) $. In this case, there exists a unique isotropic orbit 
by the  diagonal action. This isotropic orbit has the following geometric interpretation:
the map $-\mathrm{id}$ of  $\mathfrak{u}$ takes the 
orbit $\Ad\left( U\right) \left( iH\right) $ to the orbit $\mathrm{%
	Ad}\left( U\right) \left( -iH\right) =\Ad\left( U\right) \left(
i\sigma \left( H\right) \right) $ defining a diffeomorphism between the  flag $%
\mathbb{F}_{\Theta }$ and its dual flag $\mathbb{F}_{\Theta ^{\ast }}$. Since the 
isotropic orbit of the diagonal action is given by 
$$
\{\left( X,-X\right) \in \Ad\left( U\right) \left( iH\right) \times 
\Ad\left( U\right) \left( -iH\right) :X\in \Ad\left(
U\right) \left( iH\right) \}
,$$%
we conclude that this isotropic orbit is the graph of the 
diffeomorphism defined by the antipodal map  $-\mathrm{id}$. Such graph
has dimension  $\dim \mathbb{F}_{\Theta }=\dim \mathbb{%
	F}_{\Theta ^{\ast }}$. Therefore, the isotropic orbit is in fact
Lagrangian.\\ 
In conclusion, we have obtained the following description of 
isotropic orbits.

\begin{proposition}\label{diagonalL}
	There exists a unique isotropic orbit of the diagonal action of $U$ on 
	$\mathbb{F}_{\Theta }\times \mathbb{F}_{\Theta ^{\ast }}$. Such orbit is Lagrangian
	and is given as the graph of  $-\mathrm{id}:%
	\Ad\left( U\right) \left( iH\right) \rightarrow \Ad\left(
	U\right) \left( i\sigma \left( H\right) \right)$.
\end{proposition}

\subsection{Shifted diagonals as Lagrangians}
Variations of the 
diagonal action may be obtained by the action  on a product  of flags 
by subgroups of the type
$$
\Delta ^{m}=\left\{\left( u,mum^{-1}\right) \in U\times U:u\in U\right\}
$$
for any given $m\in U$. The Lie algebra of $\Delta ^{m}$ is 
$$
\Delta _{\mathfrak{u}}^{m}=\{\left( X,\Ad\left( m\right) X\right)
\in \mathfrak{u}\times \mathfrak{u}:X\in \mathfrak{u}\}.
$$
This is isomorphic to  $\mathfrak{u}$ and its orthogonal complement is given by
$$
\left( \Delta _{\mathfrak{u}}^{m}\right) ^{\bot }=\{\left( X,-\Ad%
\left( m\right) X\right) \in \mathfrak{u}\times \mathfrak{u}:X\in \mathfrak{u%
}\}.
$$%
Thus, the diagonal action by the subgroup 
$\Delta ^{m}$ has an isotropic orbit in the flag $\Ad%
\left( U\times U\right) \left( iH_{1},iH_{2}\right) $ if and only if such orbit
contains the elements of the form $\left( X,-\Ad\left(
m\right) X\right) $. This happens if and only if there exist elements $u_{1},u_{2}\in U$
such that $\Ad\left( u_{1}\right) (iH_{1})=-\Ad\left(m^{-1}u_2\right) (iH_{2})$, that is, $-iH_{1}=\Ad\left( v\right)
\left( iH_{2}\right) $ where $v=u_{1}^{-1}m^{-1}u_2$. Therefore, similarly 
to what happen for the  diagonal action, such isotropic orbits only exist in the products
$\mathbb{F}_{\Theta }\times \mathbb{F}_{\Theta ^{\ast }}$
of dual flags, namely when $-iH_{1}$ belongs to the adjoint orbit of  $%
iH_{2}$.  Now, it is simple to see that the elements of the orbit of $\left( iH_{1},-\Ad\left( m\right)
(iH_{1})\right) $ by the diagonal action $\Delta ^{m}$ have the form $\left( X,-\Ad\left( m\right) (X)\right)$. Reciprocally, given an element $\left( X,-\Ad\left( m\right) (X)\right)$ in $\Ad(U)(iH_1)\times\Ad(U)(iH_2)$, if $X=\Ad(u)(iH_1)$ with $u\in U$ then
\begin{eqnarray*}
	\left( X,-\Ad\left( m\right) (X)\right) &=& \left( \Ad\left( u\right) \left( iH_{1}\right) ,-\Ad%
	\left( m\right) \Ad\left( u\right) \left( iH_{1}\right) \right)\\
	&=& \left( \Ad\left(
	u\right) \left( iH_{1}\right) ,\Ad\left( mum^{-1}\right) \left(-\Ad%
	\left( m\right) \left( iH_{1}\right) \right)\right).
\end{eqnarray*}%
This means that $\left( X,-\Ad\left( m\right) (X)\right)$ belongs to the orbit of $\left( iH_{1},-\Ad\left( m\right)
(iH_{1})\right)$ by the diagonal action of $\Delta ^{m}$ . Therefore, here there also exists a unique isotropic orbit by the diagonal action of the subgroup $\Delta ^{m}$ and this is given by the graph of $-\Ad\left( m\right) $. 
\begin{remark}
	A direct computation allows us to show that an arbitrary product of flags $\mathbb{F}_{\Theta _{1}}\times\cdots\times \mathbb{F}_{\Theta
		_{p+1}} =\Ad\left( U^{p+1}\right) \left( iH_{1},\cdots, iH_{p+1}\right)$ has an isotropic orbit by the diagonal action of the subgroup
	$$\{ (u,m_1um_1^{-1},\cdots, m_pum_p^{-1})\in U^{p+1}:\ u\in U\},$$
	if and only if there exist $u_{1},\cdots,u_{p+1}\in U$ such that 
	$$\displaystyle \Ad(u_1)(iH_1)+\sum_{j=1}^{p}\Ad(m_j^{-1}u_{j+1})(iH_{j+1})=0.$$
\end{remark}
Summing up,

\begin{proposition}\label{shifted11}
	Inside the product $\mathbb{F}_{\Theta }\times \mathbb{F}_{\Theta ^{\ast }}$,
	for each $m\in U$,
	there exists a unique isotropic orbit of the diagonal action by the subgroup
	$\Delta ^{m}=\{\left(
	u,mum^{-1}\right) :u\in U\}$. Such an orbit is 
	Lagrangian and it is given by the graph of the map $-\Ad\left( m\right) :%
	\Ad\left( U\right) \left( iH\right) \rightarrow \Ad\left(
	U\right) \left( i\sigma \left( H\right) \right)$.
\end{proposition}
Next we prove that each pair of Lagrangian orbits in $\mathbb{F}_{\Theta }\times \mathbb{F}_{\Theta ^{\ast }}$ of the previous proposition are Hamiltonian isotopic.	
\begin{lemma} Let $m_1,m_2\in U$ with $m_1\neq m_2$. Denote by
	$$\mathcal{L}_1=\Gamma\left\lbrace-\Ad\left( m_1\right) :%
	\Ad\left( U\right) \left( iH\right) \rightarrow \Ad\left(
	U\right) \left( i\sigma \left( H\right) \right)\right\rbrace\quad\textnormal{and}$$
	$$\mathcal{L}_2=\Gamma\left\lbrace-\Ad\left( m_2\right) :%
	\Ad\left( U\right) \left( iH\right) \rightarrow \Ad\left(
	U\right) \left( i\sigma \left( H\right) \right)\right\rbrace,$$
	the Lagrangian orbits in $M=\mathbb{F}_{\Theta }\times \mathbb{F}_{\Theta ^{\ast }}$ by the diagonal action of the subgroups $\Delta^{m_1}$ and $\Delta^{m_2}$, respectively. Then $\mathcal{L}_1$ is Hamiltonian isotopic to $\mathcal{L}_2$.
\end{lemma}

\begin{proof}
	Since $U$ is compact and connected, the exponential map $e:\mathfrak{u}\to U$ is surjective (see \cite[p. 243]{smgrplie}). Therefore, for $m_2 m_1^{-1}$, there exists $X\in \mathfrak{u}$ such that $e^X=m_2 m_1^{-1}$. We define $\varphi:[0,1]\times M\to M$ as
	$$ \varphi(t,(\Ad\left( u_1\right) \left( iH\right) ,\Ad%
	\left( u_2\right) \left(-iH\right))=(\Ad\left( u_1\right) \left( iH\right) ,\Ad%
	\left(e^{tX} u_2\right) \left(-iH\right)),$$
	for all $t\in [0,1]$. As the KKS symplectic form is $\Ad$-invariant and we consider in $M$ the product symplectic form (KKS symplectic form in each coordinate $\omega:=\mathrm{p}_1^*\omega_1+\mathrm{p}_2^*\omega_2$) we have that $\varphi_t$ is a symplectomorphism for all $t\in [0,1]$. Moreover 
	$$\varphi(0,(\Ad\left( u_1\right) \left( iH\right) ,\Ad%
	\left( u_2\right) \left(-iH\right))=(\Ad\left( u_1\right) \left( iH\right) ,\Ad%
	\left( u_2\right) \left(-iH\right)),$$
	that is, $\varphi_0=\mathrm{id}_M$, and for $\left( \Ad\left( u\right) \left( iH\right) ,-\Ad%
	\left( m_1u\right) \left( iH\right) \right)\in \mathcal{L}_1$ we get
	\begin{eqnarray*}
		\varphi(1,	
		\left( \Ad\left( u\right) \left( iH\right) ,\Ad%
		\left( m_1u\right) \left(-iH\right) \right)) & = & \left( \Ad\left( u\right) \left( iH\right) ,\Ad%
		\left(e^X m_1u\right) \left(-iH\right) \right)\\
		& = & \left( \Ad\left( u\right) \left( iH\right) ,-\Ad%
		\left(m_2 u\right) \left(iH\right) \right)\in \mathcal{L}_2.
	\end{eqnarray*}
	That is, $\varphi(1,\mathcal{L}_1)=\mathcal{L}_2$. Thus, $\varphi$ is a symplectic isotopy which deforms the Lagrangian orbit $\mathcal{L}_1$ to the Lagrangian orbit $\mathcal{L}_2$.\\
	
	Now, let us see that $\varphi$ is actually a Hamiltonian isotopy. Recall that each of our flag manifolds here is endow with the KKS symplectic form and for each $X\in \mathfrak{u}$ the Hamitonian vector field  $%
	\widetilde{X}=\ad\left( X\right) $ has 
	Hamiltonian function  $\hat{\mu}(X)\left( x\right) =\langle x,X\rangle $.  As $\Ad(e^{tX})=e^{t\ad(X)}$, we have that
	$$\frac{d}{dt}\varphi_t(\Ad\left( u_1\right) \left( iH\right) ,\Ad%
	\left( u_2\right) \left(-iH\right)) =  \frac{d}{dt}(\Ad\left( u_1\right) \left( iH\right) ,\Ad%
	\left(e^{tX} u_2\right) \left(-iH\right))$$
	$$=\left(0,\frac{d}{dt}e^{t\ad(X)} \left(\Ad (u_2) \left(-iH\right)\right)\right) =  (0, \ad(X)\circ e^{t\ad(X)}(\Ad (u_2) (-iH)))$$
	$$=(\widetilde{0},\widetilde{X}) \left(\Ad\left( u_1\right) \left( iH\right) ,\Ad%
	\left( e^{tX}u_2\right) \left(-iH\right)\right)$$
	$$= (\widetilde{0},\widetilde{X})(\varphi_t(\Ad\left( u_1\right) \left( iH\right) ,\Ad%
	\left( u_2\right) \left(-iH\right))).$$
	
	That is, $\frac{d}{dt}\varphi_t= (\widetilde{0},\widetilde{X})\circ\varphi_t$ for all $t\in [0,1]$. Hence, as $(\widetilde{0},\widetilde{X})$ is a Hamiltonian vector field on $M$ it follows that $\varphi$ is a Hamiltonian isotopy. Thus, $\mathcal{L}_1$ and $\mathcal{L}_2$ are Hamiltonian isotopic. 
\end{proof}
\begin{theorem}\label{HamiltonianIsotopic}
	All Lagrangian orbits in $\mathbb{F}_{\Theta }\times \mathbb{F}_{\Theta ^{\ast }}$ of Proposition 4 belong the same Hamiltonian isotopy class.
\end{theorem}

\section{Infinitesimally tight immersions}\label{imersions}
Y. G. Oh in	\cite{Oh1} studied tight Lagrangian submanifolds of $\mathbb C \mathbb P^n$ and posed the question of classifying  all possible tight Lagrangian submanifolds in  Hermitian symmetric spaces. In particular, he asked
whether  the real forms are the only possible tight Lagrangian submanifolds. Later, C. Gorodski and F. Podest\`a classified those compact tight Lagrangian submanifolds which have the $\mathbb{Z}_2$-homology of a sphere in the case of irreducible compact homogeneous K\"ahler manifolds  \cite{GP}. The concept of tightness 
has applications to the problem of Hamiltonian volume minimization. For instance, Kleiner and Oh showed
that the standard $\mathbb R \mathbb P^n$ inside $\mathbb C\mathbb P^n$ is tight and 
has the least volume among its 
Hamiltonian deformations. Some features of intersection theory of Lagrangian submanifolds including some Lagrangian orbits in complex hyperquadrics can be found in \cite{bedu,IMMO}.\\ 
Here we explore the concept of \emph{infinitesimally tight} which we show to  be equivalent to the notion of locally tight. We give examples of infinitesimally tight Lagrangians, and we prove that the Lagrangians orbits by the diagonal and shifted diagonal actions in the product of two flags, found in the previous sections, are infinitesimally tight. Let $G$ be a Lie group and  $M$ a homogeneous space together with 
a $G$-invariant symplectic form $\omega $, that is, the action 
of $G$ on $\left( M,\omega \right) $ is symplectic.

\begin{definition}\label{tight}\cite{Oh1}
	A Lagrangian submanifold $\mathcal{L}$ in $M$ is called {\it globally tight}
	(respectively {\it locally tight}) if for all  $g\in G$
	(respectively $g$ near the identity) such that $\mathcal{L}$ intersects $%
	g\left( \mathcal{L}\right) $ transversally, we have
	$$
	\#\left( \mathcal{L}\cap g\left( \mathcal{L}\right) \right) =\mathrm{SB}\left( \mathcal{L},\mathbb{Z}%
	_{2}\right)
	$$%
	where $\#\left( \cdot \right) $ is the number of intersection points, and $\mathrm{SB}%
	\left( \mathcal{L},\mathbb{Z}_{2}\right) $ is the sum of the  $\mathbb{Z}_{2}$ Betti numbers, that is, 
	the sums of the dimensions of the homologies of $\mathcal{L}$ with
	$\mathbb{Z}_{2}$ coefficients.
\end{definition}

\begin{remark}
	In \cite{irsak} a
	Lagrangian submanifold $\mathcal{L}$ of a  K\"{a}hler manifold $M$ is called {\it globally tight} 
	(or {\it locally tight}) if the conditions of Definition \ref{tight}
	are satisfied for isometries of $M$. The definition of  
	\cite{irsak} is directed to Hermitian symmetric spaces, this is why 
	it refers to isometries of  $M$. Definition \ref{tight} 
	adapts the concept of   \cite{irsak} and considers more general 
	symmetric homogeneous  spaces.
\end{remark}
\begin{remark} The equality appearing in  Definition \ref{tight} is the lower bound of the inequality of the Arnold--Givental conjecture,
	namely $\#(\mathcal{L} \cap g(\mathcal{L})) \geq \mathrm{SB}(\mathcal{L},\mathbb Z_2)$. The conjecture has been proven in many cases, see for instance \cite{Oh2} and the survey \cite{MO}.
\end{remark}
Denote by $\widetilde{X}$ the fundamental vector field associated to an element $X\in\mathfrak{g}$.
\begin{definition}
	\label{definfintight}
	Let $\mathcal{L}$ be a submanifold of $M$. 
	An element  $X\in 
	\mathfrak{g}$ is called  {\it transversal} to $\mathcal{L}$ if
	it satisfies the following 2 conditions
	\begin{enumerate}
		\item  for any $x\in \mathcal{L}$, if  $\widetilde{X}\left( x\right) \in
		T_{x}\mathcal{L}$, then $\widetilde{X}\left( x\right) =0$, and  
		\item  the set
		$$
		f_{\mathcal{L}}\left( X\right) =\{x\in \mathcal{L}:0=\widetilde{X}\left( x\right) \in T_{x}\mathcal{L}\}
		$$%
		is finite. 
	\end{enumerate}
	In other words, $\widetilde{X}$ is only  tangent to $\mathcal{L}$ at most at finitely many points
	where it vanishes. A Lagrangian submanifold  $\mathcal{L}$ in a homogeneous space $M$ is called 
	{\it  infinitesimally tight} if the equality 
	$$
	\#\left( f_{\mathcal{L}}\left( X\right) \right) =\mathrm{SB}\left( \mathcal{L},\mathbb{Z}%
	_{2}\right),
	$$
	is satisfied for any  $X\in 
	\mathfrak{g}$ such that $\widetilde{X}$ is transversal to $\mathcal{L}$. 
\end{definition}

\begin{example} Let $\mathcal{L}$ be a maximal circle in the sphere $S^2$
	considered as a homogeneous manifold, then $\mathcal{L}$ is locally tight, globally tight and 
	infinitesimally tight. This happens because the Hamiltonian vector fields on 
	$S^2$ are generated by the moments of rotation around the $x$,$y$ and $z$ axis.
\end{example}

\begin{theorem}\label{lvi}
	Let $G$ be a Lie group and $M$ a homogeneous space together with 
	a $G$-invariant symplectic form $\omega $.	Then a Lagrangian submanifold $\mathcal{L}\subset M$ is infinitesimally tight if and only if $\mathcal{L}$ is locally tight.
\end{theorem}

\begin{proof}
	Let $\iota\colon \mathcal{L} \rightarrow M$ be a Lagrangian submanifold of $M$.
	By Weinstein's neighborhood Theorem \cite{We}, to decide whether $\mathcal{L}$ is locally tight or infinitesimally tight,
	we may assume that $M = T^*\mathcal{L}$. 
	Let us denote by $V_x:=\pi^{-1}(x)$ the (vertical) fibre of $\pi\colon T^*\mathcal{L} \rightarrow \mathcal{L}$ at $x$. Let $X \in \mathfrak g$ and  $\widetilde{X}$ the corresponding 
	fundamental  vector field. At each point $x\in \mathcal{L}$ we may write $$\widetilde{X}(x) = i(x) \oplus v(x),$$
	where $i(x)=(\iota^*\widetilde{X})(x) \in T_x\mathcal{L}$, $v(x) \in TV_x$, and $\oplus$ denotes metric orthogonal.\\
	Assume $\widetilde{X}$ is transversal to $\mathcal{L}$, then by definition, if $v(x) =0$, (that is, 
	if $\widetilde{X}(x)$ is tangent to $\mathcal{L}$), then we also have that $i(x)=0$. So that the zeros of 
	$\widetilde{Z}\vert_\mathcal{L} $ and the zeros of $v$ coincide. Let $t<<0$. Since $g=\exp tZ$ is close to the identity, then the submanifold
	$g(\mathcal{L})$
	intersects $\mathcal{L} $ at the points where $v$ vanishes. 
	But by the 
	assumption of transversality these are precisely the points
	where $\widetilde{X}$ vanishes. It is important to note that in this case the flow is determined by means of the exponential map of $G$. Now, assuming 
	that $\mathcal{L}$ is infinitesimally tight, it then follows that
	$\# (\mathcal{L} \cap g(\mathcal{L})) = \# f_\mathcal{L}(X) = \mathrm{SB}\left( \mathcal{L},\mathbb{Z}_{2}\right) $
	so that $\mathcal{L}$ is locally tight.\\
	Conversely, assume that $\mathcal{L}$ is locally tight. For each $x\in \mathcal{L}$ we follow the integral curve of $\widetilde{X}$
	until time $\epsilon<<0$ and call the new point $\mathcal{L}'(x)$. Then for small $\epsilon$, 
	the set of all such points 
	$\mathcal{L}'(x)$ with $x \in \mathcal{L}$ forms a new Lagrangian $\mathcal{L}'$ (here, the flow of every fundamental vector field determines a symplectomorphism. In fact, when we assume the existence of a moment map, as consequence of \ref{moment1} this symplectomorphism is actually Hamiltonian), then $\mathcal{L}'$ 
	is in fact a section of 
	$T^*\mathcal{L}$. Therefore, we have that 
	$\mathcal{L}' $ intersects the zero section $\mathcal{L}$ precisely at the points $x \in\mathcal{L}$ where 
	$\widetilde{X}$ vanishes, so that, assuming $\mathcal{L}$ is locally tight, we get
	$ \# f_\mathcal{L}(X)= \# (\mathcal{L} \cap g(\mathcal{L})) = \mathrm{SB}\left( \mathcal{L},\mathbb{Z}_{2}\right) $. Thus, $\mathcal{L}$ is infinitesimally tight. 
\end{proof}

\begin{remark} It is simple see that the conclusion of Theorem \ref{lvi} also holds true if $M$ is a symplectic manifold and $G$ acts on $M$ by symplectomorphisms. 
\end{remark}


\begin{example} \label{example-tight-inf} We consider the case of $\mathfrak{u}\left( 2\right) \subset \mathfrak{su}%
	\left( 3\right) $.
	Example \ref{example-su3} of section  \ref{secexflagcplx} considers the subgroup $%
	U_{H}\approx \mathrm{U}\left( 2\right) $ inside  $U=\mathrm{SU}\left(
	3\right) $ which has  Lie algebra given by 
	the matrices
	$$
	\mathfrak{u}_{H}=\left\{\left( 
	\begin{array}{cc}
	it & 0 \\ 
	0 & A%
	\end{array}%
	\right) :A\in \mathfrak{u}\left( 2\right) ,it+\mathrm{tr}A=0\right\}.
	$$%
	A  Lagrangian orbit $\mathcal{L}$ of this group occurs only inside the maximal flag $\mathbb{F}%
	\left( 1,2\right) $, in which case $\mathcal{L}=S^{3}$ in the space of 
	matrices ($\approx \mathbb{C}^{2}$), and we have
	$$
	\left( \mathfrak{u}_{H}\right) ^{\bot }=\left\{X_{\beta }=\left( 
	\begin{array}{cc}
	0 & \beta \\ 
	-\overline{\beta }^{T} & 0%
	\end{array}%
	\right) :\beta =\left( z_{1},z_{2}\right) \in \mathbb{C}^{2}\right\}.
	$$%
	More precisely, $\mathcal{L}=S^{3}=\Ad\left( U\right) \left( iH_{0}\right)
	\cap \left( \mathfrak{u}_{H}\right) ^{\bot }$ where $iH_{0}=\mathrm{diag}%
	\{i,0,-i\}$.
	If $X_{\beta },X_{\gamma }\in \left( \mathfrak{u}_{H}\right) ^{\bot }$, then 
	\begin{eqnarray*}
		\left[ X_{\beta },X_{\gamma }\right] &=&\left( 
		\begin{array}{cc}
			-\beta \overline{\gamma }^{T}+\gamma \overline{\beta }^{T} & 0 \\ 
			0 & -\overline{\beta }^{\im T}\gamma +\overline{\gamma }^{T}\beta%
		\end{array}%
		\right) \\
		&=&\left( 
		\begin{array}{cc}
			i\im\gamma \overline{\beta }^{T} & 0 \\ 
			0 & -\overline{\beta }^{T}\gamma +\overline{\gamma }^{T}\beta%
		\end{array}%
		\right) \in \mathfrak{u}_{H}.
	\end{eqnarray*}%
	In particular, if $x=X_{\gamma }\in \mathcal{L}$ and $X_{\beta }\in \left( \mathfrak{u}%
	_{H}\right) ^{\bot }$ then $\widetilde{X}_{\beta }\left( x\right) =%
	\ad\left( X_{\beta }\right) \left( x\right) \in \mathfrak{u}_{H}$ and $%
	\widetilde{X}_{\beta }\left( x\right) =0$ if and only if $\gamma \overline{%
		\beta }^{T}$ is  real. This happens if and only if  $\gamma $ is a
	real multiple of  $\pm \beta $. Therefore, any  $0\neq X_{\beta }\in \left( 
	\mathfrak{u}_{H}\right) ^{\bot }$ is  transversal to $\mathcal{L}$ (in the sense of 
	Definition \ref{definfintight}) and has singularities at the antipodal 
	points $\mathbb{R}X_{\beta }\cap S^{3}$.\\
	On the other hand, if $Y\in \mathfrak{u}\left( 2\right) =\mathfrak{u}_{H}$ then
	$\widetilde{Y}$ is  tangent to $\mathcal{L}=S^{3}$ and consequently is not
	transversal. Finally, if  $Z=X_{\beta }+Y$ with $X_{\beta }\neq 0\neq Y\in 
	\mathfrak{u}\left( 2\right) $ then $\widetilde{Z}\left( x\right) \notin
	T_{x}\mathcal{L}$ if $x$ is not a singularity of $\widetilde{X}_{\beta }$
	since in such a case $\widetilde{X}_{\beta }\left( x\right) \notin T_{x}\mathcal{L}$ is $%
	\widetilde{Y}\left( x\right) \in T_{x}\mathcal{L}$. Thus,  $Z=X_{\beta }+Y$ is 
	transversal to $\mathcal{L}$ if and only if the  singularities of $\widetilde{X}_{\beta
	} $ are also singularities of  $\widetilde{Y}$, which in turn occurs 
	if and only if  $\left[ Y,X_{\beta }\right] =0$, given that 
	the singularities of  $X_{\beta }$ belong to  $\mathbb{R}X_{\beta }$. The 
	condition  $\left[ Y,X_{\beta }\right] =0$ still holds true when $Y=0$, that 
	is, $Z=X_{\beta }$. Summing up, $Z=X_{\beta }+Y$ is  transversal if and only if  $X_{\beta }\neq 0$ and 
	$\left[ Y,X_{\beta }\right] =0$. Therefore, we conclude that  transversal elements have precisely  2 singularities, thus in agreement 
	with the sum of  Betti numbers of  $\mathcal{L}=S^{3}$. 
	
	
\end{example}
Hence, we have  obtained
\begin{proposition}
	The Lagrangian orbit  $\mathcal{L}=S^{3}$ of \, $\mathrm{U}\left( 2\right) $ in the  flag $%
	\mathbb{F}\left( 1,2\right) $ is infinitesimally tight.
\end{proposition}

\begin{example}(Diagonal action)
	In section \ref{secdiag} we established that 
	the set 
	$$
	\mathcal{L}=\{\left( X,-X\right) \in \Ad\left( U\right) \left( iH\right)
	\times \Ad\left( U\right) \left( -iH\right) :X\in \Ad\left(
	U\right) \left( iH\right) \}
	$$%
	inside the product of a flag $\mathbb{F}_{\Theta }=\Ad\left( U\right)
	\left( iH\right) $ by its  dual $\mathbb{F}_{\Theta ^{\ast }}=\Ad%
	\left( U\right) \left( -iH\right) $ is the unique  Lagrangian orbit of the 
	diagonal action of $U$ in  $\mathbb{F}_{\Theta }\times \mathbb{F}%
	_{\Theta ^{\ast }}$. This orbit is infinitesimally tight when  $%
	\mathbb{F}_{\Theta }\times \mathbb{F}_{\Theta ^{\ast }}=\Ad\left(
	U\right) \left( iH\right) \times \Ad\left( U\right) \left(
	-iH\right) $ is regarded as an adjoint orbit of $U\times U$.\\
	To verify it, the first step is to find elements  $\left(
	Y,Z\right) \in \mathfrak{u\times u}$ which are transversal to  $\mathcal{L}$ in 
	the sense of Definition \ref{definfintight}. The  tangent space to  $%
	\mathcal{L}$ at $\left( x,y\right) \in \mathcal{L}$ is given by  
	$$
	T_{\left( x,y\right) }\mathcal{L}=\left\{\left( \widetilde{A}\left( x\right) ,\widetilde{A}%
	\left( y\right) \right) :A\in \mathfrak{u}\right\}.
	$$%
	So that the  tangent space $T_{\left( X,-X\right) }\mathcal{L}$ of the obit in the
	product is 
	$$
	T_{\left( X,-X\right) }\mathcal{L}=\{\left( \left[ A,X\right] ,-\left[ A,X\right]
	\right) :A\in \mathfrak{u}\}.
	$$%
	Accordingly, $\left( \widetilde{Y}\left( x\right) ,\widetilde{Z}\left(
	y\right) \right) \in T_{\left( X,-X\right) }\mathcal{L}$ if and only if there exists $A\in 
	\mathfrak{u}$ such that $\left[ Y,X\right] =\left[ A,X\right] $ and $\left[ Z,X%
	\right] =-\left[ A,X\right] $, that is, $\left[ Y,X\right] =-\left[ Z,X%
	\right] $, or alternatively, precisely when  $X$ is a singularity of  $%
	\widetilde{Y+Z}$ in the  flag $\mathbb{F}_{\Theta }=\Ad\left( U\right)
	\left( iH\right) $. Therefore, the first condition for  transversality says that 
	$\widetilde{Y+Z}$ has a finite number of singularities over the  flag $%
	\mathbb{F}_{\Theta }$. The second condition  requires $\left( 
	\widetilde{Y}\left( x\right) ,\widetilde{Z}\left( y\right) \right) =0$ when  $%
	\left( \widetilde{Y}\left( x\right) ,\widetilde{Z}\left( y\right) \right)
	\in T_{\left( X,-X\right) }\mathcal{L}$, which means that  $\left[ Y,X\right] =-\left[ Z,X%
	\right] =0$. Consequently, a pair $\left( Y,Z\right) \in \mathfrak{u}\times \mathfrak{u}$ is 
	transversal  to the Lagrangian orbit  $\mathcal{L}$ of the  diagonal action on  $%
	\mathbb{F}_{\Theta }\times \mathbb{F}_{\Theta ^{\ast }}$ if and only if  $Y+Z$
	has a finite number of singularities on  $\mathbb{F}_{\Theta }$, which are 
	singularities of  $Y$  as well as singularities of $Z$.\\
	Now, if  $A\in \mathfrak{u}$ then $\widetilde{A}=\ad\left(
	A\right) $ on an adjoint orbit $\Ad\left( U\right) \left(
	iH\right) $ is the  Hamiltonian vector field of the height function $\hat{\mu}(A)\left(
	x\right) =\langle A,x\rangle $. Therefore, the singularities of  $\widetilde{A}$
	coincide with the singularities of  $\hat{\mu}(A)$. There are finitely many singularities
	if and only if  $A$ is a regular element. In such case, the singularities are 
	parametrized by the Weyl group $\mathcal{W}$. The number of singularities 
	equals the cardinality of  $\mathcal{W}/\mathcal{W}_{\Theta }$ (compare \cite{GGSM1}, Proposition 2.4).
\end{example}

Thus, we have obtained the following  
final characterization of the elements of  $\mathfrak{u}\times 
\mathfrak{u}$ that are transversal to  $\mathcal{L}$.

\begin{proposition}
	Let  $\mathbb{F}_{\Theta }=\mathrm{%
		Ad}\left( U\right) \left( iH\right) $. 
	A pair $\left( Y,Z\right) \in \mathfrak{u}\times \mathfrak{u}$ is 
	transversal to a  Lagrangian orbit $\mathcal{L}$ of the  diagonal action on  $%
	\mathbb{F}_{\Theta }\times \mathbb{F}_{\Theta ^{\ast }}$ if and only if  $Y+Z$
	is  regular in $\mathfrak{u}$ and both $Y$ and  $Z$ belong to the 
	intersection of the centralizers  $w\left( iH\right) $, for all  $w$
	in the  Weyl group $\mathcal{W}$.
\end{proposition}

\begin{corollary}
	The  Lagrangian orbit  $\mathcal{L}$ of the diagonal action  is  infinitesimally tight in $\mathbb{F}_{\Theta }\times \mathbb{F}_{\Theta ^{\ast }}$.
\end{corollary}
\begin{proof}
	The Lagrangian orbit $\mathcal{L}$ is diffeomorphic to  $\mathbb{F}_{\Theta }$ and the sum of the
	Betti numbers of  $\mathbb{F}_{\Theta }$ (also with $\mathbb{Z}$ coefficients)
	is the number of fixed points (singularities) of a  regular element, which is 
	the cardinality of the quotient $\mathcal{W}/\mathcal{W}_{\Theta }$.
\end{proof}

A similar argument gives us a more general collection of infinitesimally tight Lagrangians:

\begin{corollary}\label{shifted}
	The  Lagrangian orbits  of type  
	$$\Gamma\left\lbrace-\Ad\left( m\right) :%
	\Ad\left( U\right) \left( iH\right) \rightarrow \Ad\left(
	U\right) \left( i\sigma \left( H\right) \right)\right\rbrace$$ corresponding to the  shifted diagonals $\Delta^m$  are  infinitesimally tight in $\mathbb{F}_{\Theta }\times \mathbb{F}_{\Theta ^{\ast }}$.
\end{corollary}

\appendix
\section{KKS symplectic form on orthogonal groups}\label{OrthogonalLieGroups}
Transferring the KKS symplectic form  to  a adjoint orbit of a semisimple Lie group it is a well known construction, see for instance \cite{ABB}. We  describe the more general case of the KKS symplectic form on the adjoint orbit of an orthogonal Lie group. Let $G$ be a real connected Lie group of dimension $n$ and $\mathfrak{g}$ its Lie algebra. For every $g\in G$, we denote by $L_g:G\to G$ (resp. $R_g:G\to G$) the left (resp. right) multiplication of $g$ on $G$. Recall that a {\it pseudo-metric} on a smooth manifold is a symmetric and nondegenerate $(0,2)$ tensor field with constant index. A pseudo-metric $\langle\cdot,\cdot\rangle$ over $G$ is called {\it bi-invariant} if 
$L_g$ and $R_g$ are isometries of $(G,\langle\cdot,\cdot\rangle)$ for all $g\in G$. If $\textnormal{Ad}:G\to\textnormal{GL}(\mathfrak{g})$ and $\textnormal{ad}:\mathfrak{g}\to\mathfrak{gl}(\mathfrak{g})$ denote the adjoint representations of $G$ and $\mathfrak{g}$, respectively, then to have a bi-invariant pseudo-metric $\langle\cdot,\cdot\rangle$ on $G$ is equivalent to having a scalar product $\langle\cdot,\cdot\rangle_0:\mathfrak{g}\times \mathfrak{g}\to \mathbb{R}$\footnote{If $V$ is a finite-dimensional space a scalar product on $V$ is an application $\mu_0:V\times V\to\mathbb{R}$ which is bilinear symmetric and nondegenerate.} such that any of the following statements are satisfied
\begin{enumerate}
	\item $\textnormal{Ad}_g:\mathfrak{g}\to \mathfrak{g}$ is a linear isometry of $(\mathfrak{g},\langle\cdot,\cdot\rangle_0)$ for all $g\in G$, that is
	\begin{equation}\label{Inv1}
	\langle\textnormal{Ad}(g) X,\textnormal{Ad}(g)Y\rangle_0=\langle X,Y\rangle_0,\qquad g\in G,\quad X,Y\in\mathfrak{g}.
	\end{equation}
	\item $\textnormal{ad}(X):\mathfrak{g}\to \mathfrak{g}$ is an infinitesimal isometry of $(\mathfrak{g},\langle\cdot,\cdot\rangle_0)$ for all $X\in\mathfrak{g}$, that is
	\begin{equation}\label{Inv2}
	\langle \textnormal{ad}(X)(Y),Z\rangle_0+\langle Y,\textnormal{ad}(X)Z\rangle_0=0,\qquad X,Y,Z\in\mathfrak{g}.
	\end{equation}
\end{enumerate}
A scalar product which satisfies the identity (\ref{Inv2}) is called \emph{invariant}.
\begin{definition}
	\begin{enumerate}\label{ort}
		\item The pair $(G,\langle\cdot,\cdot\rangle)$ is called an \emph{orthogonal Lie group} if $\langle\cdot,\cdot\rangle$ is a bi-invariant pseudo-metric on $G$.
		\item If $\mathfrak{g}$ is a finite-dimensional Lie algebra and $\langle\cdot,\cdot\rangle_0$ is an invariant scalar product on $\mathfrak{g}$ the pair $(\mathfrak{g},\langle\cdot,\cdot\rangle_0)$ is called an \emph{orthogonal Lie algebra}.
	\end{enumerate}
\end{definition}
It is simple to see that there exists a bijective correspondence between simply connected orthogonal Lie groups and orthogonal Lie algebras. If $\langle\cdot,\cdot\rangle_0:\mathfrak{g}\times\mathfrak{g}\to \mathbb{R}$ is an invariant scalar product, then
$$\langle u(g),v(g)\rangle_g:=\langle(\textnormal{d}L_{g^{-1}})_g(u(g)),(\textnormal{d}L_{g^{-1}})_g(v(g))\rangle_0\qquad g\in G,
$$
is a bi-invariant pseudo-metric on $G$. 
\begin{example}
	Examples of orthogonal Lie groups are the compact Lie groups, semisimple Lie groups, the cotangent bundle of a Lie group, and the $\lambda$-oscillator groups. For the last example see \cite{M}.
\end{example}
Recall that the tangent space to the adjoint orbit $\Ad(G)(H)$ of an element $H\in\mathfrak{g}$ is given by $T_{X_0}(\Ad(G)(H))=\lbrace [X,X_0]:\ X\in \mathfrak{g} \rbrace$ where $\widetilde{X}=\textnormal{ad}(X)$ is the fundamental vector field by the adjoint action associated to $X\in\mathfrak{g}$. On the other hand, if $\mathfrak{g}^\ast$ denotes the dual vector space of $\mathfrak{g}$ and $\Ad^\ast:G\to \textnormal{GL}(\mathfrak{g}^\ast)$ the coadjoint representation of $G$, the tangent space to the coadjoint orbit $\Ad^\ast(G)(\alpha)$ of an element $\alpha\in\mathfrak{g}^\ast$ is $T_{\beta}(\Ad^\ast(G)(\alpha))=\lbrace -\beta\circ \textnormal{ad}(X):\ X\in \mathfrak{g} \rbrace$. Here $\widetilde{X}^\ast=\textnormal{ad}^\ast(X)$ is the fundamental vector field by the coadjoint action of $X\in\mathfrak{g}$. Let $(G,\langle\cdot,\cdot\rangle)$ be an orthogonal Lie group and $(\mathfrak{g},\langle\cdot,\cdot\rangle_e)$ its respective orthogonal Lie algebra. As $\langle\cdot,\cdot\rangle_e:\mathfrak{g}\times\mathfrak{g}\to\mathbb{R}$ is nondegenerate, the map $\varphi:\mathfrak{g}\to\mathfrak{g}^\ast$ defined by $\varphi(X)=\langle X,\cdot\rangle_e$ is a linear isomorphism. If $\alpha\in\mathfrak{g}^\ast$, we denote by $X_\alpha$ the only $X_\alpha\in \mathfrak{g}$ such that $\alpha(\cdot)=\varphi(X_\alpha)$.
\begin{lemma}\label{Lemma1}
	If $(G,\langle\cdot,\cdot\rangle)$ is an orthogonal Lie group, then $\varphi:\mathfrak{g}\to\mathfrak{g}^\ast$ is equivariant with respect the adjoint and coadjoint actions. Consequently, the adjoint and coadjoint representations of $\mathfrak{g}$ are isomorphic.
\end{lemma}
\begin{proof}
	The first claim is an immediate consequence of identity (\ref{Inv1}). If we put $g=e^{tX}$ in the formula $\textnormal{Ad}^\ast(g)\circ \varphi=\varphi\circ \textnormal{Ad}(g)$ and apply derivative at $t=0$ to both sides of the last formula we get that $\textnormal{ad}^\ast(X)\circ \varphi=\varphi\circ \textnormal{ad}(X)$ for all $X \in \mathfrak{g}$.
\end{proof}
Assume $(G,\langle\cdot,\cdot\rangle)$ is an orthogonal Lie group. It is simple to see that Lemma \ref{Lemma1} implies that $\varphi^{-1}\left(\Ad^\ast(G)(\alpha)\right)=\Ad(G)(X_\alpha)$ for all element $\alpha$ of $\mathfrak{g}^\ast$. Therefore, we have that
\begin{proposition}
	If $(G,\langle\cdot,\cdot\rangle)$ is an orthogonal Lie group, $\varphi:\mathfrak{g}\to \mathfrak{g}^\ast$ maps adjoint orbits diffeomorphically and  $G$-equivariantly onto coadjoint orbits. Moreover, there exists a symplectic form $\Omega$ on the adjoint orbit $\Ad(G)(X)$ such that the adjoint action restricted to $\Ad(G)(X)$ determines a symplectic action of $G$.
\end{proposition}
\begin{proof}
	If we consider the restriction of the adjoint action and the coadjoint action of $G$ on $\Ad(G)(X)$ and $\Ad^\ast(G)(\varphi(X))$, respectively, the first claim is clear. On the other hand, it is well known that for all $\alpha\in\mathfrak{g}^\ast$, the coadjoint orbit $\Ad^\ast(G)(\alpha)$ is a symplectic manifold with the symplectic for given by
	$$
	\omega_\beta(\textnormal{ad}^\ast(Y)(\beta),\textnormal{ad}^\ast(Y)(\beta))=\beta([X,Y]).$$
	Moreover, it is holds that for all $g\in G$, the map $\textnormal{Ad}^\ast(g)|_{\Ad^\ast(G)(\alpha)}$ is a symplectomorphism of $(\Ad^\ast(G)(\alpha),\omega)$. Therefore, the pullback of $\omega$ by $\varphi$ induces a symplectic form $\Omega$ on the adjoint orbit $\Ad(G)(X_\alpha)$  which satisfies that $\textnormal{Ad}(g)|_{\Ad(G)(X_\alpha)}$ is a symplectomorphism of $(\Ad(G)(X_\alpha),\Omega)$ for all $g\in G$ since by Lemma \ref{Lemma1} we have that
	$$		\left(\textnormal{Ad}(g)\right)^\ast \Omega = \left(\textnormal{Ad}(g)\right)^\ast (\varphi^\ast \omega)  =  (\varphi\circ \textnormal{Ad}(g))^\ast \omega = (\textnormal{Ad}^\ast(g)\circ \varphi)^\ast \omega = \varphi^\ast \omega = \Omega.$$
	Explicitly, the symplectic form on $\Ad(G)(X)$ is given by
	$$\omega_{X_0}(\textnormal{ad}(Y)(X_0),\textnormal{ad}(Z)(X_0))=\varphi(X_0)([Y,Z])=\langle X_0,[X,Y]\rangle_e.$$
\end{proof}
An immediate consequence  is
\begin{corollary}
	The map $\mu:\Ad(G)(X)\to \mathfrak{g}^\ast$ defined by $\mu(X_0)=\varphi(X_0)$ defines an $\Ad^\ast$-equivariant moment map for $(\Ad(G)(X_\alpha),\Omega)$.
\end{corollary}
\begin{proof}
	For all $Y\in \mathfrak{g}$ and $X_0\in \Ad(G)(X)$ we get
	\begin{eqnarray*}
		\textnormal{d}\widehat{\mu}(Y)_{X_0}(\textnormal{ad}(Z)(X_0)) & = & \frac{d}{dt}\left\langle\textnormal{Ad}\left(e^{tZ}\right)(X_0),Y\right\rangle_e|_{t=0}=\langle Y,[Z,X_0]\rangle_e\\
		& = & \langle X_0,[Y,Z]\rangle_e=\Omega_{X_0}(\textnormal{ad}(Y)(X_0),\textnormal{ad}(Z)(X_0)).
	\end{eqnarray*}
	That is, $\textnormal{d}\widehat{\mu}(Y)=\iota_{\widetilde{Y}}\Omega$. Lemma \ref{Lemma1} implies that this moment map is $\textnormal{Ad}^\ast$-equivariant since
	$$\mu(\textnormal{Ad}(g)(X_0))=\varphi(\textnormal{Ad}(g)(X_0))=\textnormal{Ad}^\ast(g)(\varphi(X_0))=\textnormal{Ad}^\ast(g)(\mu(X_0)).$$
\end{proof}
\section{Open questions}
In this appendix we establish some problems whose interest arises from the results proved in this paper and the motivation was expressed previously.
\begin{problem*} 
	Find a characterization of 
	Lagrangian orbits  in terms of the moment map  $\mu $ (analogous to Proposition \ref{anill}) 
	in the case when  $\mu $ is not equivariant for the  
	adjoint representation, but only with respect to an affine representation.
\end{problem*}

\begin{problem*}
	Find the flags where $L$ admits an
	isotropic orbit, or more specifically,  a Lagrangian one. Equivalently,  find 
	the  \textquotedblleft types\textquotedblright\ of elements of  $%
	\mathfrak{u}$ in the orthogonal complement of  $\left( l^{\prime }\right) ^{\bot }$.
\end{problem*}

\begin{problem*} \label{pairs} Determine the pairs $H_{1},H_{2}\in \mathfrak{a}^{+}$ such that
	$U_{H_{1}}$ has  isotropic or Lagrangian orbit in   $\mathbb{F}%
	_{\Theta_{2}}$. This problem was partially solved in \cite{bedori}, who
	classified the linear compact groups that have Lagrangian orbits for
	projective spaces (therefore not just subgroups that are  centralizers of the
	tori). Beware of not jumping to the false conclusion of thinking that 
	\textquotedblleft  because the flag is a projective submanifold
	then the classification given by  \cite{bedori}, solves also the case of 
	flags\textquotedblright . This does not solve the problem, since a 
	Lagrangian submanifold of a   projective submanifold is only isotropic in 
	projective space, given that it has less than half the  dimension. 
\end{problem*}

\begin{problem*} Problem \ref{pairs} is likely to be related to the following question 
	about semisimple algebras:  let $\mathfrak{p}$ be a parabolic subalgebra. 
	Determine the types of nilpotent orbits that intersect 
	the nilradical $\mathfrak{n}$ of  $\mathfrak{p}$.
\end{problem*}

\begin{problem*} \label{weins} When an orbit  $U_{H}x$ is  Lagrangian its
	dimension is half of the dimension of the orbit $Z_{H}x$ which is open. A
	natural question is whether  $Z_{H}x$ is the  cotangent bundle of  $U_{H}x$,
	that is, whether there exists a  symplectic diffeomorphism between $Z_{H}x$ and $%
	T^{\ast }\left( U_{H}x\right) $ where the cotangent bundle is considered with the 
	canonical symplectic form. This problem is inspired in the  
	general theorem of Weinstein  \cite{We} that identifies  the cotangent bundle $T^{\ast }\mathcal{L}$ of a
	Lagrangian submanifold  $\mathcal{L}$ with a  tubular neighborhood of  $\mathcal{L}$.
\end{problem*}

\begin{problem*} An extension of  Problem \ref{weins}  is to ask when does  it happen 
	that an orbit  $U_{H}x$ is isotropic but not  Lagrangian. We may
	expect the following situation to hold in general: a) $Z_{H}x$ is a symplectic submanifold;
	b) $U_{H}x$ is a  Lagrangian submanifold of   $Z_{H}x$ and c) $%
	Z_{H}x\approx T^{\ast }\left( U_{H}x\right) $.
\end{problem*}

\begin{problem*} Generalize all results proved in the case of product of two flags for arbitrary products of flags.
\end{problem*}

\begin{problem*}
	Generalize Example \ref{example-tight-inf} for orbits of  of the Grassmanians 
	$$\mathrm{SU}\left( n\right) /S\left( \mathrm{U}\left( k\right) \times \mathrm{%
		U}\left( n-k\right) \right).$$
\end{problem*}

\section*{Acknowledgments}
	We want to thank Severin Barmeier and Alexander Givental for enlightening discussions. We are grateful for the support given by the Network NT8 from the Office of External Activities of the Adbus Salam International Center for Theoretical Physics (ICTP), Trieste, Italy. E. Gasparim was partially supported by a Simons Associateship ICTP.

\end{document}